\documentclass[a4paper,12pt]{article}

\usepackage{amsmath,amssymb,amscd, amsthm}

\usepackage{color}
\usepackage[dvips]{graphicx} 
\usepackage{version}
\newenvironment{NB}{
\color{red}{\bf NB}. \footnotesize
}{}
\excludeversion{NB}

\newenvironment{NB2}{
\color{red}{\bf NB}. \footnotesize
}{}
\excludeversion{NB2}

\newtheorem{thm}{Theorem}[section]
\newtheorem{defn}[thm]{Definition}
\newtheorem{ex}[thm]{Example}
\newtheorem{prop}[thm]{Proposition}
\newtheorem{cor}[thm]{Corollary}
\newtheorem{lem}[thm]{Lemma}

\newtheorem{rem}[thm]{Remark}

\newcommand{\Hom}{\mathrm{Hom}}

\newcommand{\mr}[1]{{\mathrm{#1}}}

\newcommand{\mca}[1]{{\mathcal{#1}}}

\title{Wall-crossing of the motivic Donaldson-Thomas invariants}
\author{Kentaro Nagao}
\begin{document}

\maketitle

\begin{NB2}
%\begin{center}
This is a preliminary version !
%\end{center}
\end{NB2}%

\begin{abstract}
We study motivic Donaldson-Thomas invariants in the sense of \cite{behrend_bryan_szendroi}. 
A wall-crossing formula under a mutation is proved for a certain class of quivers with potentials.
The class includes the quivers for the canonical bundles of del-Pezzo surfaces with geometric helices. 
The formula is the same as \cite{ks, COHA} and is described by quantum dilogarithms.
%As a by-product, we provide an alternative proof the quantum dilogarithm identify \cite{keller-q-dilog}.
\end{abstract}

\begin{NB}
\begin{center}
%{\LARGE A note motivic identities for moduli stacks}
{\LARGE On motivic Donaldson-Thomas invariants}
\end{center}

\smallskip

\begin{flushright}
{\large Kentaro Nagao}

\smallskip

%{\large 
Graduate School of Mathematics

Nagoya University
%}
\end{flushright}

\bigskip
\end{NB}

\begin{NB}
%\begin{center}
This is a preliminary version !
%\end{center}
\end{NB}%

\section*{Introduction}
In this article we study the {\it motivic Donaldson-Thomas} (DT in short) {\it invariants} introduced in \cite{ks,behrend_bryan_szendroi}.

The DT invariant for a Calabi-Yau $3$-fold $Y$ is a counting invariant of coherent sheaves on $Y$, 
which it is introduced in \cite{thomas-dt} as a holomorphic analogue of the Casson invariant on a real $3$-manifold.
The moduli space involves a symmetric obstruction theory and a virtual fundamental cycle \cite{behrend-fantechi-intrinsic,behrend-fantechi}. 
The invariant is defined as a integration of the constant function $1$ over the virtual fundamental cycle.

The DT invariant has the other description : it coincides with the weighted Euler characteristic weighted by the Behrend function.
It is known that the moduli space of coherent sheaves on $Y$ can be locally described as the critical locus of a function which is called a {\it holomorphic Chern-Simons functional} (see \cite{joyce-song}).
%In a case when we have a grobal description 
The value of the Behrend function is given by the Euler characteristic of the {\it Milnor fiber} of the Chern-Simons functional \cite{behrend-dt}.

Following these results of Behrend, it is proposed in \cite{ks,behrend_bryan_szendroi} to study {\it motivic Milnor fiber} as a motivic version of the DT invariant so that we can get a refinement of the ordinary DT invariant by applying a suitable cohomology functor for the motivic one . 
Such a refinement has been expected in string theory \cite{RTV,dimofte-gukov,dimofte_gukov_soibelman}.

%Given a helix of sheaves on a smooth projective variety, the rolled-up helix algebra 

%We study tilting for a class of Calabi-Yau algebras associated to helices on Fano varieties.

In \cite{ks}, Kontsevich and Soibelman provided a wall-crossing formula for motivic DT invariants up to a certain identity for motivic Milnor fibers (\cite[Conjecture 4]{ks}). 
The aim of this article is to give an alternative proof of the wall-crossing formula for (\cite{behrend_bryan_szendroi}'s) motivic DT invariants in a spacial setting.

\section*{Main result}
Let $(Q,W)$ be a quiver with a potential (QP in short).
In this paper, we assume that $W$ is finite, i.e. a finite linear combination of oriented cycles.
 
Let $\mathcal{D}_{Q,W}$ be the derived category of dg modules with finite dimensional cohomologies over the (non-complete) Ginzburg's dg algebra and
$\mathrm{mod}J(Q,W)$ be the category of finite dimensional modules over the (non-complete) Jacobi algebra which is the core of the natural bounded t-structure of $\mathcal{D}_{Q,W}$.
The moduli stack of objects is canonically described as the critical locus of a function $f_W$ on a smooth stack $\mathfrak{M}_Q$. We call the function as the {\it Chern-Simons functional}.
We define the motivic DT invariant by the virtual motive $[\mathrm{crit}(f_W)]_{\mathrm{vir}}$ (Definition \ref{defn_vir}) of the critical locus of the Chern-Simons functional \footnote{The smooth stack $\mca{M}_Q$ is described as a quotient stack divided by a special algebraic group $G$. Actually we do not work on critical loci for stacks but only for varieties and define motivic invariants by the quotient of the virtual motives by $G$.}.

%For a vertex $k$, we can mutate $(Q,W)$ to get a new QP $\mu_k(Q,W)=(\mu_kQ,\mu_kW)$ \footnote{We need to assume that $Q$ has no loops or $2$-cycles and $W$ is generic.}.
For a vertex $k$ without loops, let $\mu_k(Q,W)=(Q',W')$ be the mutation in the sense of \cite{quiver-with-potentials}.
We assume that $Q'$ is the quiver mutation in the sense of Fomin-Zelevinsky and $W'$ is finite \footnote{In \cite{quiver-with-potentials}, it is shown that if the potential $W$ is generic then $Q'$ is the Fomin-Zelevinsky mutation of $Q$. Finiteness of $W$ is stronger assumption.}.
%we can mutate $(Q,W)$ to get a new QP $\mu_k(Q,W)=(Q',W')$ \footnote{We need to assume that $Q$ has no loops or $2$-cycles and $W$ is generic.}.
Keller-Yang showed that $\mathcal{D}_{Q,W}$ and $\mathcal{D}_{Q',W'}$ are equivalent \cite{dong-keller,keller-completion}. 
We want to describe the relation between the motivic DT invariant for $(Q,W)$ and the one for $(Q',W')$.

The main theorem in this paper is the following :
\begin{thm}\label{thm_main}
%\label{thm_WC}
Assume that 
\begin{itemize}
\item $(Q,W)$ has a cut (Definition \ref{defn_cut}), 
\item $k$ is a strict source of $C$ (Definition \ref{defn_strict}).
%\item $(Q,W)$ is Calabi-Yau.
%\begin{NB2}I'm sure I can remove Calabi-Yau condition.\end{NB2}%
\end{itemize}
Then we have
\[
\mca{A}_{Q',W'}
\ "="\ 
\mathbb{E}(s_k[1])\times \mca{A}_{Q,W}\times \mathbb{E}(s_k)^{-1}
\]
where $\mca{A}_{Q,W}$ is the generating series of the motivic Donaldson-Thomas invariants and $\mathbb{E}(s_k)$ is the ``motivic dilogarithm'' (Example \ref{ex_dilog}). This is an equation in the ``{\it motivic torus}'' (Definition \ref{defn_mt}). \footnote{This is an equation of infinite power series. Since we have to make it clear in which completion we work, I use the equal sign with quotation mark $"="$.} 

By taking the weight polynomial we get the following :
\[
{A}_{Q',W'}
\ "="\ 
\mathbb{E}_q(s_k[1])\times {A}_{Q,W}\times \mathbb{E}_q(s_k)^{-1}
\]
where ${A}_{Q,W}$ is the generating series of the ``refined Donaldson-Thomas invariants'' (Definition \ref{defn_rdt}) and $\mathbb{E}_q(s_k)$ is the quantum dilogarithm. 
This is an equation in the ``{\it quantum torus}'' (Definition \ref{defn_qt}).
\end{thm}

\section*{Sketch of the proof}

\noindent
{\it \underline{First step}}

\smallskip 

The first step is to show the {\it factorization property}. 
Take a stability condition, then each object $\mathrm{mod}J(Q,W)$ has the unique Harder-Narashimhan filtration. 
Types of the Harder-Narashimhan filtrations induce a filtration by open sets on the moduli stack (see \S \ref{subsec_ks52}).
In particular, each stratum is smooth. 
Using this filtration we want a formula which describes the generating function of the motivic DT invariants as the product of the generating functions of the motivic invariants of the moduli stacks of semi-stable objects. 
To get the formula, we need the following ; 
\begin{quote}
%Let $\mca{X}$ be a smooth stack, $f$ be a function on $\mca{X}$ and $\mca{Y}\subset \mca{X}$ be a smooth substack of codimension $d$. Then, 
Let ${X}$ be a smooth stack, $f$ be a function on ${X}$ and ${Y}\subset {X}$ be a smooth substack of codimension $d$. Then, 
\begin{equation}\label{eq_motivic}
\bigl[\mathrm{crit}(f)\bigr]_{\mathrm{vir}}\overset{?}{=}
\bigl[\mathrm{crit}(f|_{{X}\backslash {Y}})\bigr]_{\mathrm{vir}}+
\mathbb{L}^{-\frac{d}{2}}\cdot \bigl[\mathrm{crit}(f|_{Y})\bigr]_{\mathrm{vir}}.
%\bigl[\mathrm{crit}(f)\bigr]_{\mathrm{vir}}\overset{?}{=}
%\bigl[\mathrm{crit}(f|_{\mca{X}\backslash \mca{Y}})\bigr]_{\mathrm{vir}}+
%\mathbb{L}^d\cdot \bigl[\mathrm{crit}(f|_\mca{Y})\bigr]_{\mathrm{vir}}.
\end{equation}
\end{quote}
%In \cite{COHA}, Kontsevich-Soibelman show the factorization property for their {\it cohomological Hall algebra} by using  
%show the factorization property for their {\it cohomological Hall algebra} by using  
%Let 
%\[
%Z\colon K_0(\mathrm{mod}J(Q,W))\simeq\mathbb{Z}^{Q_0}\to \mathbb{C}
%Z\colon \mathbb{Z}^{Q_0}\to \mathbb{C}
%\]
%be a central charge (see \ref{}). 

%In \S \ref{sec_FP}, we work in a case when we can apply \cite[Theorem B.1]{behrend_bryan_szendroi}, that is, when the moduli stack involves a torus action with some good properties. 
%Then the equation \eqref{eq_motivic} directly follows \cite[Theorem B.1]{behrend_bryan_szendroi}. 
%In \S \ref{sec_FP}, we work in a case when the QP has a {\it cut} (Definition \ref{defn_cut})\footnote{}.
%If the QP has a cut, then the moduli stack involves a $\mathbb{C}^*$-action so that we can apply \cite[Theorem B.1]{behrend_bryan_szendroi}.
%In \S \ref{sec_FP}, we assume that we have a nonnegative grading
In \S \ref{sec_FP}, we assume that we have a cut $C$ of the QP $(Q,W)$, that is, a nonnegative grading
\[
%\mr{deg}\colon Q_1\to \mathbb{Z}_{\geq 0}
g_C\colon Q_1\to \mathbb{Z}_{\geq 0}
\]
such that $W$ is homogeneous of degree $1$.
Then the moduli stack involves a $\mathbb{C}^*$-action so that we can apply \cite[Theorem B.1]{behrend_bryan_szendroi} (Theorem \ref{thm_BBS}).
The equation \eqref{eq_motivic} directly follows \cite[Theorem B.1]{behrend_bryan_szendroi} (see Proposition \ref{prop_31}). 
\begin{rem}
In \cite{COHA}, Kontsevich-Soibelman introduce the {\it cohomological Hall algebra} (COHA in short) which provide another realization of a refinement of the DT invariant. 
The factorization property for the COHA is shown in \cite[\S 5]{COHA}.
For the COHA, the Thom isomorphism is the counterpart of the equation \eqref{eq_motivic}.
\end{rem}

Applying the factorization property in our setting, 
we can describe the generating function of the motivic DT invariants as the product of the generating functions of the motivic invariants of
the moduli stacks of objects in 
\[
\mathcal{S}:=\{s_k^{\oplus n}\mid n\geq 0\}
\]
and
\[
{}^\bot \mathcal{S}:=\{X\in \mathrm{mod}J(Q,W)\mid \Hom (X,s_k)=0\}
\] 
where $s_k$ is the simple $J(Q,W)$-module corresponding to the vertex $k$.
The generating function for $\mathcal{S}$ is given by the {\it quantum dilogarithm}.

\smallskip 

\smallskip 
\noindent
{\it \underline{Second step}}

\smallskip

In the same way,  
we can describe the generating function for $(Q',W')$ as the product of the generating functions of the motivic invariants of
the moduli stacks of objects in 
\[
 (\mathcal{S}')^\bot:=\{X\in \mathrm{mod}J(Q',W')\mid \Hom (s'_k,X)=0\}.
\]
and
\[
\mathcal{S'}:=\{(s'_k)^{\oplus n}\mid n\geq 0\}
\] 
where $s'_k$ is the simple $J(Q',W')$-module.
It is shown in \cite{dong-keller} that the derived equivalence is given by {\it tilting} with respect to the simple module $s_k$, that is, in the derived category we have
\[
\mathcal{S}'=\mathcal{S}[1],\quad 
 (\mathcal{S}')^\bot={}^\bot \mathcal{S} \quad (\text{see Figure \ref{fig:tilting}}).
\]
\begin{figure}[h]
  \centering
  \input{tilting.tpc}
  \caption{$\mathrm{mod}J(Q,W)$ and $\mathrm{mod}J(Q',W')$}
  \label{fig:tilting}
\end{figure}
Now, we get two Chern-Simons functionals which realize the moduli stack of objects in $(\mathcal{S}')^\bot={}^\bot \mathcal{S}$ as the critical loci ; 
%one is induced from $W$ and the other is induced from $\mu_k W$. 
one is the restriction of $f_W$ and the other is the restriction of $f_{\mu_k W}$. 
A priori, the virtual motive depends not only on the scheme structure of the critical locus but also on the choice of the Chern-Simons functional.
So we need to show the following : 
%\begin{quote}
\begin{align}
&\text{The virtual motives of the moduli stack of objects in $(\mathcal{S}')^\bot={}^\bot \mathcal{S}$ defined}\notag \\
&\text{by $f_{W}$ and $f_{W'}$ coincide.}\label{coincide}
\end{align}
%\end{quote}
Combine \eqref{coincide} with the arguments above, we can describe the relation between the motivic DT invariants for $(Q,W)$ and for $(Q',W')$ in terms of the quantum dilogarithm. 

We prove \eqref{coincide} under the assumption in Theorem \ref{thm_main} (Proposition \ref{prop_44}). 
The proof consists of the following two steps :

\smallskip

%In \S \ref{sec_WC}, we assume that the QP has a cut and, moreover, is Calabi-Yau.
\noindent (A) 
%In \S \ref{sec_WC}, we assume the three conditions in Theorem \ref{thm_main}.
%we work in a case when the QP is $3$-dimensional Calabi-Yau and has a {\it cut} of the $QP$.
%If the QP has a cut $C$, then the moduli stack involves so that we can apply \cite[Theorem B.1]{behrend_bryan_szendroi}.
%Moreover, 
By taking the torus fixed part of the Jacobi algebra $J(Q,W)$ we can define the {\it truncated Jacobi algebra} $J(Q,W)_C$ (\S \ref{subsec_21}) and we have the following identity (Theorem \ref{prop_44}) : 
\begin{align}
&\text{The virtual motive of moduli stack of $J(Q,W)$-modules }\notag \\
%&\text{ of the moduli stack of objects ``supported on the torus fixed locus''.}\label{reduction}
%&\text{with the (not virtual !) motive of the moduli stack of}\notag \\
%&\text{objects $J(Q,W)_C$-modules.}
&=\text{ the motive of the moduli stack of $J(Q,W)_C$-modules.}\label{reduction}
\end{align}
This is a generalization of \cite[Equation (2.4)]{behrend_bryan_szendroi} and \cite[Theorem 9.5]{hua}.

\smallskip 

%Assume we take a cut $C'$ of the QP $(Q',W')$ such that $J(Q',W')_{C'}$ is derived equivalent to $J(Q,W)$, then we can see \eqref{coincide} by showing the identity of the moduli stack of $J(Q,W)_C$-modules and $J(Q',W')_{C'}$-modules.
\noindent (B) If $k$ is a {\it strict source} (Definition \ref{defn_strict}), we can take a cut $C'$ of $(Q',W')$ and show an identity between the moduli stack of $J(Q,W)_C$-modules and the one of $J(Q',W')_{C'}$-modules (Proposition \ref{prop_44}).
This makes us possible to compare the virtual motive of the moduli stack of $J(Q,W)$-modules and the one of $J(Q',W')$-modules.

\section*{Comments}
Let us itemize some applications, related topics and further directions.
Some of them will appear in the forthcoming paper.
\begin{itemize}
%\noindent(a) 
\item[(a)]
%Using Theorem \ref{thm_FP}, 
Applying Theorem \ref{thm_FP} for a product of two simply laced Dynkin quivers, we can show a quantized version of dilogarithm identity in conformal field theory \cite{nakanishi_dilog}.
More general identities has been already shown by B. Keller \cite{keller-q-dilog}.% \footnote{His argument works for arbitrary quiver with periodicity but our argument works only for products of simply laced Dynkin diagrams at this moment.}.
%\smallskip
%\noindent(b) 
\item[(b)]
In \cite{cluster-via-DT}, the author studied cluster algebras by using the ideas in Donaldson-Thomas theory.
It is expected that we can study quantum cluster algebras \cite{quantum_cluster} using motivic Donaldson-Thomas theory.
%\smallskip
%\noindent(c)
\item[(c)] 
In this paper, the result of Behrend-Bryan-Szendroi \cite[Theorem B.1]{behrend_bryan_szendroi} plays a crucial role and existence of desirable torus action is indispensable.
We want to show the same results for any generic QP in the future.
Once we get \eqref{eq_motivic} and \eqref{coincide}, then we can prove the same results immediately.
%\smallskip
%\noindent(d) 
\item[(d)]
During preparing this paper, the author was informed by Sergey Mozgovoy of his related work. 
In \cite{mozgovoy} he shows a similar result to Theorem \ref{thm_FP} over finite fields.
He uses the result of M. Reineke \cite{reineke_homomorphism}.
%\smallskip
%\noindent(e) 
\item[(e)]
During preparing this paper, the author was informed also by Balazs Szendroi of his related work. 
In \cite{szendroi_morrison}, Szendroi and A. Morrison provide a motivic version of the result of \cite{nagao-nakajima}.
As a result they realize the refined topological vertex of the generating function of motivic invariants, which has already discussed in physics (\cite{dimofte-gukov, dimofte_gukov_soibelman}). 
We can apply Theorem \ref{thm_FP} to study the wall-crossing phenomenon. \footnote{The author was informed by Andrew Morrison that he has a generalization for $\mathbb{C}\times \mathbb{C}^2/(\mathbb{Z}/n\mathbb{Z})$. }
\end{itemize}

\section*{Acknowledgement}
{\ }

\vspace{-5mm}
 
This paper is strongly influenced by ideas in \cite{ks, COHA}.

This paper has also benefited from lectures by Zheng Hua in January 2011 at Nagoya.
He explained to the author his usage of \cite[Theorem B.1]{behrend_bryan_szendroi} in his preprint \cite{hua}.
The author greatly appreciates him.

%During preparing this paper, the author was informed by Sergey Mozgovoy and Balazs Szendroi of their related works. 

%In \cite{mozgovoy}, Mozgovoy shows a similar result to Theorem \ref{thm_FP} over finite field using the result of M. Reineke \cite{reineke_homomorphism}.

%In \cite{szendroi_morrison}, Szendroi and A. Morrison provide a motivic version of the result of \cite{nagao_nakajima}. 
%We can apply Theorem \ref{thm_FP} to study the wall-crossing phenomenon.

The author thanks Osamu Iyama and Perre-Guy Plamondon for the useful discussions and comments.
The author also thanks Z. Hua, S. Mozgovoy and B. Szendroi for pointing out some mistakes in the preliminary version.

The author is supported by the Grant-in-Aid for Research Activity Start-up (No. 22840023) and for Scientific Research (S) (No. 22224001).

\section{Motivic Donaldson-Thomas invariants}
\subsection{Motivic ring}
%\begin{defn}
Let $K_0(\mathrm{Var}/\mathbb{C})$ denote 
the free abelian group on isomorphism classes of complex varieties, 
modulo relations
\[
[X] = [Z] + [U]
\]
for $Z\subset X$ a closed subvariety with complementary open subvariety $U$.
%\end{defn}
We can equip $K_0(\mathrm{Var}/\mathbb{C})$ with the structure of a commutative
ring by setting
\[
[X]\cdot[Y]=[X\times Y].
\]
%Let $K_0(\mathrm{Var}_\mathbb{C})$ be the Grothendieck ring of variteis. 
We write
\[
\mathbb{L} = [\mathbb{A}^1]\in K_0(\mathrm{Var}/\mathbb{C})
\]
for the class of the affine line.
We define the motivic ring 
\[
\mathcal{M}_\mathbb{C}:=K_0(\mathrm{Var}/\mathbb{C})[\mathbb{L}^{-1/2}]
\]
and its localization 
\[
%\mathcal{M}^{\mr{loc}}_\mathbb{C}:=
\widetilde{\mathcal{M}}_\mathbb{C}:=
\mathcal{M}_\mathbb{C}[(1-\mathbb{L}^n)^{-1}:n\geq 1].
\]
%The multiplication is given by the product of two verieties.
The following lemma is a consequence of \cite[Lemma 3.8]{bridgeland-hall}:
\begin{lem}\label{lem_comporison}
Let $X$ (resp. $Y$) be a variety with an action of a special algebraic group $G$ (resp. $H$).
Assume we have an isomorphism of stacks between $[X/G]$ and $[Y/H]$, then we have
\[
\frac{[X]}{[G]}=\frac{[Y]}{[H]}\in \widetilde{\mathcal{M}}_\mathbb{C}.
%\mathcal{M}_\mathbb{C}[(1-\mathbb{L}^n)^{-1}:n\geq 1].
\]
\end{lem}
%For a reduced veriety $S$ we define relative Grothendieck ring $K_0(\mathrm{Var}/S)$ and and relative motivic ring $\mathcal{M}_S$ in the same way (see \cite[\S 1.3]{behrend_bryan_szendroi}). 
Let $\hat{\mu}:=\lim_{\leftarrow}\mathbb{Z}/n\mathbb{Z}$ be the group of roots of unity. 
%We define relative Grothendieck ring $K_0^{\hat{\mu}}(\mathrm{Var}/S)$ of virieties with good $\hat{\mu}$-actions and the corresponding motivic ring $\mathcal{M}_S^{\hat{\mu}}$ as in \cite[\S 1.4]{behrend_bryan_szendroi}.
We define the Grothendieck group $K_0^{\hat{\mu}}(\mathrm{Var}/\mathbb{C})$ 
of varieties with good $\hat{\mu}$-actions as in \cite[\S 1.4]{behrend_bryan_szendroi}.
%and the corresponding motivic ring $\mathcal{M}^{\hat{\mu}}$ 
We define $\widetilde{\mathcal{M}}^{\hat{\mu}}_\mathbb{C}$ in the same way.

%The additive group $K_0^{\hat{\mu}}(\mathrm{Var})$ 
The additive group  $\widetilde{\mathcal{M}}^{\hat{\mu}}_\mathbb{C}$
can be endowed with an associative multiplication $\star$
using convolution involving the classes of Fermat curves \cite{denef-loeser-igusa,looijenga}. 
This product agrees with the ordinary product on 
%the subalgebra $K_0(\mathrm{Var})\subset K_0^{\hat{\mu}}(\mathrm{Var})$ of classes with trivial $\hat{\mu}$-actions, but not in general.
the subalgebra $\widetilde{\mathcal{M}}_\mathbb{C} \subset \widetilde{\mathcal{M}}^{\hat{\mu}}_\mathbb{C}$ of classes with trivial $\hat{\mu}$-actions, but not in general.

\subsection{Homomorphisms from the motivic ring}
Deligne's mixed Hodge structure on compactly supported cohomology of a variety $X$ gives rise to the $E$-polynomial homomorphism
\[
E
\colon 
K_0(\mr{Var}_\mathbb{C})
\to 
\mathbb{Z}[x, y]
\]
defined on generators by
\[
E([X];x,y)=
\sum_{p,q}
x^py^q
\sum_i
(-1)^i \dim H_{p,q}(H^i_c(X,\mathbb{Q})).
\]
This extends to a ring homomorphism
\[
E\colon 
%\mathcal{M}_\mathbb{C}[(1-\mathbb{L}^n)^{-1}:n\geq 1]
\widetilde{\mathcal{M}}_\mathbb{C}
\to
%\mathbb{Z}[x,y,(xy)^{1/2},(x^ny^n-1)^{-1};n\geq 1].
\mathbb{Q}(x^{1/2},y^{1/2})
\]
By the specialization
\[
x = y = (xy)^{1/2}= q^{1/2},
\]
we get 
\[
\mathbb{W}\colon 
%\mathcal{M}_\mathbb{C}[(1-\mathbb{L}^n)^{-1}:n\geq 1]
\widetilde{\mathcal{M}}_\mathbb{C}
%[(1-\mathbb{L}^n)^{-1}:n\geq 1]
\to
%\mathbb{Z}[q^{1/2},(q^n-1)^{-1};n\geq 1].
\mathbb{Q}(q^{1/2})
\]
\begin{ex}\label{ex_dilog}
We put
\[
\mca{T}:=\prod_{n\geq 0}\widetilde{\mca{M}}_{\mathbb{C}}\cdot e_n
\]
where $e_n$ is a formal variable satisfying $e_n\cdot e_m=e_{n+m}$. 
We put 
%$\hat{e}_n:=\mathbb{L}^{n^2/2}\cdot e_n$ and 
\[
\sum_{n\geq 0}\frac{[\mr{pt}]}{[\mr{GL}_n]\cdot \mathbb{L}^{-\frac{\dim \mr{GL}_n}{2}}}\cdot{e}_n\in \mca{T}.
\]
We call this \textup{motivic dilogarithm}.
We extend the homomorphism $\mathbb{W}$ to 
\[
\mca{T}\to T:=
\prod_{n\geq 0}
%\mathbb{Z}[q^{1/2},(q^n-1)^{-1};n\geq 1]
\mathbb{Q}(q^{1/2})
\cdot
e_n.
\]
Then the image of the motivic dilogarithm under $\mathbb{W}$ is the \textup{quantum dilogarithm (\cite{quantum-dilogarithm}) :}
\[
\sum_{n\geq 0}\frac{q^{n^2/2}}{(q^n-1)\cdots(q^n-q^{n-1})}{e}_n\in {T}.
\]
\end{ex}

\subsection{Motivic nearby and vanishing cycles}
Let $f\colon X \to \mathbb{C}$ be a regular function on a smooth variety X and let $X_0 := f^{-1}(0)$ be the central fiber.
Using arc spaces, Denef and Loeser \cite{denef-loeser,looijenga} define
%$[\varphi_f]_{X_0}\in \mathcal{M}_{X_0}^{\hat{\mu}}$, 
%the relative motivic nearby cycle of $f$. 
the motivic nearby cycle $[\varphi_f]\in \mathcal{M}_\mathbb{C}^{\hat{\mu}}$ of $f$ and 
%Using motivic integration, 
%Denef?Loeser give an explicit formula for $[\phi_f]_{X_0}$ 
%in terms of an embedded resolution of $f$. 
%We give this formula in detail in Appendix B.
the motivic vanishing cycle 
\[
%[\varphi_f]_{X_0}:=[\psi_f]_{X_0}-[X_0]_{X_0}\in \mathcal{M}_{X_0}^{\hat{\mu}}
[\varphi_f]:=[\psi_f]-[X_0]\in \mathcal{M}_{\mathbb{C}}^{\hat{\mu}}
\]
of $f$.
Note that if $f=0$, then $[\psi_0]=-[X]$.
%We will denote by $[\psi_f],[\varphi_f]\in \mathcal{M}_{\mathbb{C}}^{\hat{\mu}}$
%the absolute motivic nearby and vanishing cycles, 
%the images of the relative classes under pushforward to the point.
\begin{NB}
The following lemma, which is a special case of the motivic Thom-Sebastiani theorem (Theorem \ref{thm_thos_sebastiani}), directly follows the definition : 
\begin{lem}
Let $X$ and $Y$ be smooth varieties and $f$ be a regular function on $X$. 
Then we have
\[
[\psi_{\pi_X^*f}]=[\psi_{f}\times Y]
\]
where $\pi_X\colon X\times Y \to X$ is the natural projection.
\end{lem}
\end{NB}
\begin{NB}
\begin{lem}\label{lem_13}
Let $G$ be an algebraic group, $H$ be a subgroup of $G$, $X$ be a smooth variety with an action $H$ and $f$ be an $H$-invariant regular function on $X$.
We put $X_G:=X\times_H G$ and $f_G$ be the pull back of $f$ on $X_G$. 
Then we have
\[
\frac{[\varphi_{f_G}]}{[G]}=\frac{[\varphi_f]}{[H]}\in 
%\mathcal{M}_\mathbb{C}[(1-\mathbb{L}^n)^{-1}:n\geq 1].
\widetilde{\mathcal{M}}_\mathbb{C}.
\]
\end{lem}
\begin{proof}
We use the description of the motivic nearby cycle in terms of any embedded resolution (\cite{denef-loeser,looijenga}).
We can take an $H$-equivariant embedded resolution $Y\to X$ of $f^{-1}(0)$. 
Then $Y_G:=Y\times_H G$ gives a $G$-equivariant embedded resolution of $f_G^{-1}(0)$.
The claim directly follows the descriptions of $[\varphi_{f}]$ and $[\varphi_{f_G}]$ in terms of $Y$ and $Y_G$. 
We omit the details, since in the setting of this paper we can prove the claim using Theorem \ref{thm_BBS}.  
\end{proof}
\end{NB}
\begin{thm}[\protect{{\bf Motivic Thom-Sebastiani Theorem} \cite{denef-loeser_ts,looijenga}}]\label{thm_thom_sebastiani}
Let $f$, $g$ be regular functions on smooth varieties $X$, $Y$. Then we have
\[
[-\varphi_{f\oplus g}]
=
[-\varphi_{f}]
\star
[-\varphi_{g}]
\]
\end{thm}
\begin{NB}
\begin{cor}\label{lem_12}
Let $X$ and $Y$ be smooth varieties and $f$ be a regular function on $X$. 
Then we have
\[
[\psi_{\pi_X^*f}]=[\psi_{f}]\star [Y]
\]
where $\pi_X\colon X\times Y \to X$ is the natural projection.
\end{cor}
\end{NB}

We say that a $\mathbb{C}^*$-action on a variety $X$ is weakly circle compact 
%if the fixed point set $V^{\mathbb{C}^*}$ is compact and moreover, for all $v\in V$, the limit $\lim_{\lambda \to 0}\lambda\cdot y$ exists.
if, for all $x\in X$, the limit $\lim_{\lambda \to 0}\lambda\cdot x$ exists.
\begin{thm}[\protect{\cite[Theorem B.1]{behrend_bryan_szendroi}}]\label{thm_BBS}
Let $f\colon X\to \mathbb{C}$ be a regular morphism on a smooth
quasi-projective complex variety. 
%Let $Z = \{\mathrm{d}f = 0\}$ be the degeneracy locus of $f$.
% and let $Z_\mathrm{aff}\subset X_\mathrm{aff}$ be the affinization of $Z$ and $X$ respectively.
Assume that there exists an action of a connected complex torus $T$
on $X$ so that $f$ is $T$-equivariant with respect to a primitive character
 $\chi\colon T\to \mathbb{C}^*$, namely $f(t\cdot x)=\chi(t)f(x)$ for all $x\in X$ and $t\in T$.
We further assume that there exists a one parameter subgroup $\mathbb{C}^*\subset T$ such that the induced action is weakly circle compact. 
Then the motivic nearby cycle class $[\psi_f]$ is in $\mathcal{M}_\mathbb{C}\subset\mathcal{M}_\mathbb{C}^{\hat{\mu}}$ and is equal to $[X_1] = [f^{-1}(1)]$.
Consequently the motivic vanishing cycle class $[\varphi_f]$ is given by
$[\varphi_f] = [f^{-1}(1)] - [f^{-1}(0)]$.
%If we further assume that $X_0$ is reduced then $[\varphi f ]_{Z_\mathrm{aff}}$ , the motivic vanishing cycle, considered as a relative class on $Z_\mathrm{aff}$, lies in the subring ${M_\mathrm{aff}}\subset \mathcal{M}^{\hat{\mu}}_{Z_\mathrm{aff}}$.
\end{thm}
\begin{rem}
In \textup{\cite[Theorem B.1]{behrend_bryan_szendroi}} they assume that, moreover, the fixed point set $X^{\mathbb{C}^*}$ is compact. 
As they themselves mention in \textup{\cite[pp15 l9-10]{behrend_bryan_szendroi}}, this assumption is not necessary. 
\end{rem}

\subsection{Virtual motives of critical loci}
Let $f\colon X\to \mathbb{C}$ be a regular function on a smooth variety $X$, 
%and let $Z = \{\mathrm{d}f = 0\}\subset X$ be its degeneracy locus. 
and let $\mathrm{crit}(f) = \{df = 0\}\subset X$ be its degeneracy locus. 
\begin{defn}\label{defn_vir}
%We define the relative virtual motive of $\mathrm{crit}(f)$ to be
%\begin{equation}\label{eq_vir}
%[\mathrm{crit}(f)]_\mathrm{relvir} := \mathbb{L}^{\frac{\dim X}{2}}
%[\varphi_f]_Z\in \mathcal{M}^\mu_{\mathrm{crit}(f)}
%\end{equation}
%and the absolute virtual motive of $\mathrm{crit}(f)$ to be
We define the virtual motive of $\mathrm{crit}(f)$ to be
%\begin{equation}\label{eq_vir}
\[
[\mathrm{crit}(f)]_\mathrm{vir} := -\mathbb{L}^{-\frac{\dim X}{2}}
[\varphi_f]\in \mathcal{M}^\mu_\mathbb{C}.
\]
%\end{equation}
%the pushforward of the relative virtual motive $[\mathrm{crit}(f)]_\mathrm{relvir}$ to the absolute motivic ring.
\end{defn}
\begin{rem}
The virtual motive may depend not only on the scheme structure of the critical locus but also on the presentation as a critical locus.
\end{rem}
We a smooth variety $X$, we use the following notation :
\[
[X]_\mr{vir}:=[\mathrm{crit}(0\colon X\to \mathbb{C})]_\mathrm{vir}
=\mathbb{L}^{-\frac{\dim X}{2}}[X].
\]

%\subsection{Virtual motives of moduli spaces}
\subsection{Motivic Donaldson-Thomas invariants}
Throughout this paper we assume that a quiver is finite and has no loops and oriented $2$-cycles.
Let $Q$ be a quiver,
$Q_0$ denote the set of vertices of $Q$ and 
$Q_1$ denote the set of arrows of $Q$.
For an arrow $e\in Q_1$, we denote by $t(e)\in Q_0$ (resp. $h(e)\in Q_0$) the vertex at which $e$ starts (resp. ends). 
Take a dimension vector $\mathbf{v}=(v_i)\in (\mathbb{Z}_{\geq 0})^{Q_0}$ and put $V_i=\mathbb{C}^{v_i}$. 
We define
\[
\mathrm{M}(Q;\mathbf{v}):=\bigoplus_{e\in {Q_1}}
\mathrm{Hom}(V_{t(e)},V_{h(e)})
\]
and 
\[
G(\mathbf{v}):=\prod_{i\in Q_0}\mathrm{GL}(V_i).
\]
Note that $G(\mathbf{v})$ naturally acts on $\mathrm{M}(Q;\mathbf{v})$ and the quotient gives the moduli stack of representations of $Q$ with dimension vectors $\mathbf{v}$.
Let $\chi_Q \colon \mathbb{Z}^{Q_0} \times \mathbb{Z}^{Q_0} \to \mathbb{Z}$ be the bilinear form \footnote{This is the Euler form on the Grothendieck group of the category of finite-dimensional representations
of $Q$.} given by 
\[
\chi_Q(\mathbf{v},\mathbf{v}'):=
-\sum_{i,j\in Q_0}Q_{ij}v_iv'_j+
\sum_{i\in Q_0}v_iv'_i.
\]
Then we have
\[
\dim\mathrm{M}(Q;\mathbf{v})
-
\dim G(\mathbf{v})
=
-\chi_Q(\mathbf{v},\mathbf{v}).
\]

Let $W$ be a potential, that is, a finite linear combination of cyclic paths in $Q$. 
Let $f_{W,\mathbf{v}}$ be the $G(\mathbf{v})$-invariant function on $\mathrm{M}(Q;\mathbf{v})$ defined by taking the trace of the map associated to the potential $W$.
A point in the critical locus $\mathrm{crit}(f_{W,\mathbf{v}})$ gives a $J(Q,W)$-module and the quotient stack 
\[
\bigl[\mathrm{crit}(f_{W,\mathbf{v}})/G(\mathbf{v})\bigr]
\]
gives the moduli stack of $J(Q,W)$-modules with dimension vectors $\mathbf{v}$.
\footnote{In this sense, the function $f_{W,\mathbf{v}}$ is called a Chern-Simons functional.}

\begin{defn}
For $(Q,W)$ and $\mathbf{v}$, we define {\it motivic Donaldson-Thomas invariant} by 
\[
\mathfrak{M}_{\mathrm{vir}}(Q,W;\mathbf{v}):=
\frac{[\mathrm{crit}(f_{W,\mathbf{v}})]_\mathrm{vir}}{[G(\mathbf{v})]_\mr{vir}}
%=
%-\frac{[\psi_{f_{W,\mathbf{v}}}]}{[G(\mathbf{v})]}
%\cdot
%\mathbb{L}^{-\frac{\chi_Q(\mathbf{v},\mathbf{v})}{2}}
\in 
%\mathcal{M}^{\hat{\mu}}_\mathbb{C}[(1-\mathbb{L}^n)^{-1}:n\geq 1].
\widetilde{\mathcal{M}}^{\hat{\mu}}_\mathbb{C}.
\]
\end{defn}

%\section{Truncated Jacobian and motivic identity}
\section{Cut of a QP and truncated Jacobian}
\subsection{Cut of a QP}\label{subsec_21}
%\subsection{definition}
Let $(Q,W)$ be a QP. 
To each subset $C\subset Q_1$ we associate a grading $g_C$ on $Q$ by
\[
g_C(a) = 
\begin{cases}
1 & a \in C,\\
0 & a \notin C.
\end{cases}
\]
Denote by $Q_C$, the subquiver of $Q$ with the vertex set $Q_0$ and the arrow set $Q_1\backslash C$.
\begin{defn}[\protect{\cite[\S 3]{herschend-iyama}}]\label{defn_cut}
A subset $C\subset Q_1$ is called a \textup{cut} if $W$ is homogeneous of degree $1$ with respect to $g_C$.
\end{defn}
If $C$ is a cut, then $g_C$ induces a grading on $J(Q,W)$ as well. 
The degree $0$ part of $J(Q,W)$ is denoted by $J(Q,W)_C$ 
and called the {\it truncated Jacobian algebra}. 
We have
\begin{align*}\label{eq_t_Jacobian}
J(Q,W)_C
&=
J(Q,W)\big/\langle C\rangle\\
&=
\mathbb{C}Q_C\big/\langle \partial_aW\mid a\in C\rangle.
\end{align*}

%\subsection{examples}
%\begin{enumerate}
%\item[(1)] Dimer model and perfect matching
%\item[(2)]
%\end{enumerate}

%In \cite[\S 5]{herschend-iyama}, they show some examples of cuts.
Here we will show two examples of cuts.
\subsection{Example(1) : bipartite graph and perfect matching}\label{subsec_ex1}
%This example is studied in \cite{Ishii-ueda2}.\Gamma
The following example is studied by \cite{Ishii_Ueda2}.

Let $\Sigma$ be a real $2$-dimensional oriented manifold and $\Gamma$ be a bipartite graph on $\Sigma$, that is, $\Gamma$ is a triple $(B,R,E)$ of disjoint finite subsets $B$ (the set of blue vertices) and $R$ (the set of red vertices) of $\Sigma$ and a set of $1$-cells $E$ such that 
\begin{itemize}
\item any two elements of $E$ do not intersect in their interiors. 
\item each element of $E$ connects one element in $B$ and another element in $R$.
\end{itemize}

%\footnote{$G$ is a triple $(B,W,E)$ of disjoint finite subsets $B$ and $W$ of $\Sigma$ and a set of $1$-cells $E$ such that
%\begin{itemize}
%\item any two elements of $E$ do not intersect in their interiers, and
%\item the compliment $\Sigma\backslash (B\cap W\cap\bigcap_{e\in E} e)$ is cont%ractible
%\end{itemize}
%}

We take the dual graph of $\Gamma$. 
For each element $e$ in $E$, we define the orientation of the dual edge $\hat{e}$ so that $\hat{e}$ crosses with $e$ keeping the blue boundary of $e$ on the right hand side. Let $Q_\Gamma$ denote the resulting quiver.
%Note that $(Q_\Gamma)_1=E$.
%We define the quiver $Q_{\Gamma}$ as follows: 
%the set of vertices $(Q_{\Gamma})_0$ is 
For $b\in B$ (resp. $r\in R$), let $w_b$ (resp. $w_r$) be the minimal cyclic path in $Q_\Gamma$ which goes around $b$ (resp. $r$) clockwise (resp. anti-clockwise). 
We put
\[
w_\Gamma:=\sum_{b\in B}w_b-\sum_{r\in R}w_r.
\]

\begin{ex}\label{ex1}
Let $\Gamma$ be the bipartite graph on a torus in the left of Figure \ref{fig1}.
The corresponding quiver $Q_\Gamma$ is given in the right of Figure \ref{fig1}.
The potential $W_\Gamma$ is given by 
\[
W_\Gamma=a_1b_1c_1d_1-a_1b_2c_1d_2-a_2b_1c_2d_1+a_2b_2c_2d_2.
\]
The quiver with potential is known to be derived equivalent to the quotient stack $[(\mathcal{O}_{\mathbb{P}^1}(-1)\oplus\mathcal{O}_{\mathbb{P}^1}(-1))/(\mathbb{Z}/2\mathbb{Z})]$ where $\mathbb{Z}/2\mathbb{Z}\subset \mr{SL}(\mathbb{C},2)$ acts fiberwise.
\begin{figure}[htbp]
  \centering
  \input{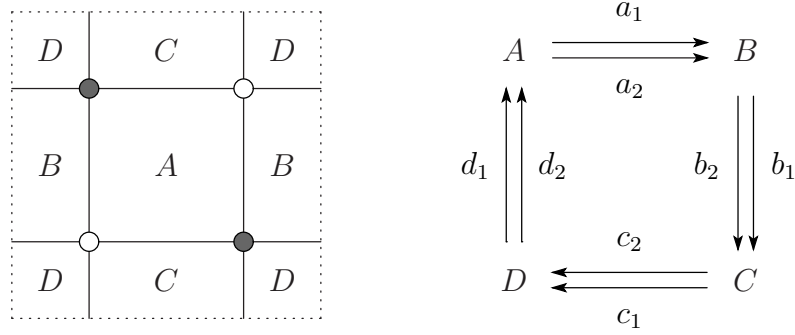}
  \caption{an example of a bipartite graph and the quiver}
  \label{fig1}
\end{figure}
\end{ex}

A {\it perfect matching} $P$ is a subset of $E$ such that each element $v\in B\cup R$ there exists exactly one element in $P$ which has $v$ as its boundary.
It is easy to check that any perfect matching $P$, as a subset of $(Q_\Gamma)_1=E$, gives a cut of the QP $(Q_\Gamma,W_\Gamma)$.

\begin{ex}
Let $P$ be the perfect matching in Figure \ref{fig2}. 
The corresponding cut is $\{a_1,a_2\}$.
\begin{figure}[htbp]
  \centering
  \input{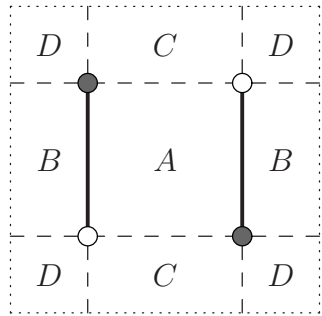}
  \caption{an example of a perfect matching}
  \label{fig2}
\end{figure}
\end{ex}

\subsection{Example(2) Geometric helices on Del Pezzo surfaces}\label{subsec_ex2}
The following example is studied by \cite{bridgeland-stern}.

Let $Y$ be a Del-Pezzo surface and let $(E_i)_{i=1,...,n}$ be a full exceptional collection on $Y$. 
Put $\mathbb{E} = \bigoplus_{i=1}^n E_i$ and define $A(\mathbb{E}) := \mr{End}(\mathbb{E})$.
We put
\[
\mathbb{H}(\mathbb{E})=(E_i)_{i\in \mathbb{Z}}
:=(\ldots, \omega_Y\otimes E_n,E_1,\ldots,E_n,\omega_Y^{-1}\otimes E_1,\ldots).
\]
and assume that
%satisfies the following conditions:
\begin{itemize}
\item $(E_i,\ldots,E_{i+N-1})$ is an exceptional collection on $Y$ for any $i$, and
\item $\mathrm{Hom}^k(E_i,E_j)=0$ for any $k\neq 0$ and any $i<j$.
\end{itemize}
Such a sequence is $(E_i)_{i\in \mathbb{Z}}$ is called a {\it geometric helix} (\cite{bondal-polishchuk}).
%A {\it geometric helix} of type $N$ on $Y$ is 
%a sequence $\mathbb{E}=(E_i)_{i\in\mathbb{Z}}$ such that 
%\begin{itemize}
%\item $E_{i+N}=E_i\otimes \omega_Y^{-1}$, 
%\item $(E_i,\ldots,E_{i+N-1})$ is an exceptional collection on $Y$ for any $i$, and
%\item $\mathrm{Hom}^k(E_i,E_j)=0$ for any $k\neq 0$ and any $i<j$.
%\end{itemize}
The {\it rolled up helix algebra} is the $\mathbb{Z}$-graded algebra
\[
{B}(\mathbb{H}) = 
\bigoplus_{k\in\mathbb{Z}}
\mathrm{Hom}(E,\omega_Y^{-k}\otimes E)
\]
with the obvious multiplication.

The following theorem is proved in \cite[\S 6.9]{keller-completion} and \cite[Appendix A]{VdB-example}.
\begin{thm}\label{thm-helix}
There is a QP $(Q,W)$ and a cut $C$ such that
\[
B(\mathbb{H})\simeq J(Q,W)
,\quad 
A(\mathbb{E})\simeq J(Q,W)_C.
\]
\end{thm}
\begin{ex}\label{ex3}
We take $Y:=\mathbb{P}_1\times \mathbb{P}_1$ 
and a geometric helix
\[
\ldots,
\mathcal{O},
\mathcal{O}(1,0),
\mathcal{O}(0,1),
\mathcal{O}(1,1),
\mathcal{O}(2,2)=\omega^{-1}_Y\otimes \mathcal{O},\ldots
\]
where we put 
$\mathcal{O}(a,b)=
\pi_1^*(\mathcal{O}_{\mathbb{P}_1}(a))
\otimes
\pi_2^*(\mathcal{O}_{\mathbb{P}_1}(b))$.
We take the quiver $Q$ in Figure \ref{fig3} and the potential 
\[
W:= \sum_{i,j\in\{1,2\}}U_{ij}(t_is_j-S_jT_i),
\]
then $(Q,W)$ satisfies the conditions in Theorem \ref{thm-helix}.

\begin{figure}[htbp]
  \centering
  \input{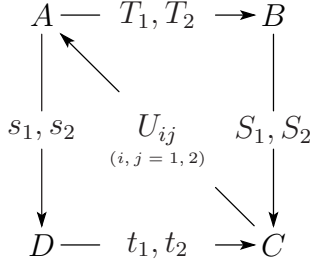}
  \caption{quiver for $\mathbb{P}_1\times \mathbb{P}_1$}
  \label{fig3}
\end{figure}

Note that $\mathbb{P}_1\times \mathbb{P}_1$ gives a crepant resolution of the quotient singularity $(\mathcal{O}_{\mathbb{P}^1}(-1)\oplus\mathcal{O}_{\mathbb{P}^1}(-1))/(\mathbb{Z}/2\mathbb{Z})$. 
In fact, the QP in this example is obtained by mutating the one in Example \ref{ex1} at the vertex $D$ and they are derived equivalent.
\end{ex}

\subsection{Mutation of QP and cut}
%\begin{NB}
%underconstruction
%\end{NB}%
Extending Fomin and Zelevinsky mutations of quivers \cite{fomin-zelevinsky1}, Derksen, Weyman, and Zelevinsky have introduced the notion of mutation of QPs in \cite{quiver-with-potentials}.
As pointed out in \cite{graded_mutation}, we can extend the definition to the graded setting. 
%In particular, the mutation at a strict source is given as follows : 
Let $(Q,W,d)$ be a $\mathbb{Z}$-graded quiver with a homogeneous potential of degree $r$ and $k$ be vertex of $Q$.
% without incident loops or $2$-cycles.
We define $\widetilde{\mu}^L_k(Q,W,d) = (\widetilde{Q},\widetilde{W},\widetilde{d})$ the left mutation of $(Q,W,d)$ at vertex $k$ as follows : 
\begin{itemize}
\item[(1)]
the new quiver $\widetilde{Q}$ is defined as follows :
\begin{itemize}
\item[(a)]
for any subquiver $u \overset{a}{\longrightarrow} k \overset{b}{\longrightarrow} v$ with $k$, $u$ and $v$
pairwise different vertices, we
add an arrow $[ba]\colon u\to v$;
\item[(b)]
we replace all arrows $a$ incident with $k$ by an arrow $a^*$ in the opposite direction.
\end{itemize}
\item[(2)]
The new potential $\widetilde{W}$ is defined by the sum $[W] + \Delta$ where $[W]$ is
formed from the potential $W$ replacing all compositions $ba$ through the vertex $k$ by the new arrows $[ba]$, and where $\Delta$ is the sum $\sum a^*b^*[ba]$.
\item[(3)]
The new degree $\widetilde{d}$ is defined as follows :
\begin{itemize}
\item[(a)]$\widetilde{d}(a) = d(a)$ for $a$ not incident to $k$ ;
\item[(b)]$\widetilde{d}([ba]) = d(b) + d(a)$ for a composition $ba$ passing through $k$ ;
%\item[(c)]$\widetilde{d}(a^*)=d(a)+r$ if the target of $a$ is $k$;
%\item[(d)]$d'(b^*)=d(b)$ if the source of $b$ is $k$.
\item[(c)]$\widetilde{d}(a^*)=-d(a)+r$ if $t(a)=k$;
\item[(d)]$\widetilde{d}(b^*)=-d(b)$ if the source of $s(b)=k$.
\end{itemize}
\end{itemize}
%\begin{itemize}
%\item[(4)]
By the graded version (\cite[Theorem 6.4]{graded_mutation}) of the splitting theorem \cite[Theorem 4.6]{quiver-with-potentials}, any graded QP $(Q,W,d)$ has a direct sum decomposition
\[
%(Q,W,d)=(Q^{\mathrm{red}},W^{\mathrm{red}},d^{\mathrm{red}})\oplus (Q^{\mathrm{triv}},W^{\mathrm{triv}},d^{\mathrm{triv}})
(Q,W,d)=(Q,W,d)^{\mathrm{red}}\oplus (Q,W,d)^{\mathrm{triv}}
\]
into a reduced graded QP and a trivial graded QP.
The decomposition is unique up to graded right equivalence. 
%\item[(5)]

Assume that the potential $W$ is generic in the sense of \cite{quiver-with-potentials}. 
Then the underlying quiver of the reduction $\mu^L_k(Q,W,d):=(\widetilde{Q},\widetilde{W},\widetilde{d})^\mr{red}$ of the left mutation $\widetilde{\mu}^L_k(Q,W,d) = (\widetilde{Q},\widetilde{W},\widetilde{d})$ coincides with Fomin-Zelevinsky's mutation.
\begin{defn}[\protect{\cite[Definition 6.12]{herschend-iyama}}]\label{defn_strict}
Let $C$ be a cut of a QP $(Q,W)$. 
We say that a vertex $k$ of $Q$ is a strict source (resp. sink) of $(Q,C)$ if all arrows ending (resp. starting) at $x$ belong to $C$ and all arrows starting (ending) at $x$ do not belong to $C$.
\end{defn}
\begin{ex}
\noindent \textup{(1)} Let $(Q,W)$ and $C$ be given as in \S \ref{subsec_ex1}.
A vertex is a strict source or a strict sink if and only if the vertices of the corresponding face of the bipartite graph is perfectly matched by the perfect matching. 

\noindent \textup{(2)} Let $(Q,W)$ and $C$ be given as in \S \ref{subsec_ex2}. Then the vertex corresponding to the exceptional object $E_1$ is a strict source.\end{ex}
The underlying graded quiver of $\mu^L_k(Q,W,d_C)$ is given as follows: 
\begin{itemize}
\item[(a)]
%for any subquiver $u \overset{a}{\longrightarrow} k \overset{b}{\longrightarrow} v$ with $k$, $u$ and $v$
%pairwise different vertices, add a degree $1$ arrow $[ba]\colon u\to v$;
add degree $1$ arrows $[ba]$ ;
\item[(b)]
replace $a$ with a degree $0$ arrow $a^*$ ; 
%replace all arrows $a$ incident with $k$ by a degree $0$ arrow $a^*$ in the opposite direction.
\item[(c)]
cancel $2$-cycles\footnote{Since we assume $W$ is generic, we can see any $2$-cycle has degree $1$. So this step has no ambiguity even in the graded sense.}.
\end{itemize}
%\end{itemize}
The new degree gives a cut of the mutated QP $\mu_k(Q,W)$. 
Let $\mu_kC$ denote this cut.
\begin{rem}
Given a strict source $k$, a new cut $C_k$ of $(Q,W)$ is defined \textup{(\cite[Definition 6.10]{herschend-iyama})} \footnote{They call $C_k$ the {\it cut mutation}.}.
If $(Q,W)$ is Calabi-Yau, then $J(\mu_kQ,\mu_kW)_{\mu_kC}$ is isomorphic to $J(Q,W)_{C_k}$.
\end{rem}

%By applying the graded mutation \cite{graded_mutation}, we get a cut $\mu_KC$ of the QP $\mu_k(Q,W)$. 

%\begin{defn}\textup{\cite[Definition 6.10]{herschend-iyama}}
%For a strict source $k$ of $(Q,C)$, we define a new cut $\mu_k (C)$ of $Q_1$ by removing all arrows
%in $Q$ ending at $k$ from $C$ and adding all arrows in $Q$ starting at $x$ to $C$.
%Let $C$ be a subset of $Q_1$.
%\begin{itemize}
%\item[(1)] 
%We say that a vertex $k$ of $Q$ is a strict source of $(Q,C)$ if all arrows ending at $x$ belong to $C$ and all arrows starting at $x$ %do not belong to $C$.%
%\item[(2)] For a strict source $K$ of $(Q,C)$, we define the subset $\mu_k (C)$ of $Q_1$ by removing all arrows
%in $Q$ ending at $k$ from $C$ and adding all arrows in $Q$ starting at $x$ to $C$.
%\item[(3)] Dually we define a strict sink and $\mu_k (C)$.
%\end{itemize}
%\end{defn}

\begin{NB}
%\subsection{Calabi-Yau case}
\subsection{Category equivalences}
\begin{NB2}
I'm sure that we can replace this subsection with the arguments in \cite[\S 10]{quiver-with-potentials}.
\end{NB2}%
Let $\Gamma_{Q,W}$ be the Ginzburg dg algebra and 
$\mathcal{D}_{Q,W}$ be the derived category of dg modules with finite dimensional cohomologies over $\Gamma_{Q,W}$.
Let 
\[
\Phi=\Phi_k \colon \mathcal{D}_{Q,W} \overset{\sim}{\longrightarrow} \mathcal{D}_{Q',W'}
\]
be the derived equivalence (\cite{dong-keller,keller-completion}) which induces the isomorphism $\phi=\phi_k$ of the Grothendieck group given by
\[
\phi([s_i])=
\begin{cases}
[s'_i]+Q(i,k)[s'_k] & i\neq k.\\
-[s'_k] & i=k.
\end{cases}
\]

Assume that $(Q,W)$ is Calabi-Yau, that is, $\Gamma_{Q,W}$ is concentrated on degree $0$. 
In such a case 
\begin{itemize}
\item[(1)] $\mathcal{D}_{Q,W}\simeq D^b(\mr{mod}J(Q,W))$, 
\item[(2)] $\mu_k(Q,W)$ is also Calabi-Yau (\cite{keller-completion}).
\end{itemize} 
Let $\Phi$ be the derived equivalence between $D^b(\mr{mod}J(Q,W))$ and $D^b(\mr{mod}J(\mu_kQ,\mu_kW))$. For any $V\in D^b(\mr{mod}J(Q,W))$, we have 
\[
H^j(\Phi(V))_i=0\ (j\neq 0),\quad H^0(\Phi(V))_i=V_i
\]
for $i\neq k$ and the map 
\[
[ba] \colon H^0(\Phi(V))_{s(a)} \to H^0(\Phi(V))_{t(b)}
\]
is given by the composition $b\circ a$.
%\begin{prop}
%\end{prop}
We define the following full subcategories 
\begin{align*}
\mr{mod}(J(Q,W))_k
&:=
\{V\in \mr{mod}(J(Q,W))\mid \Hom(s_k,V)=0\},\\
\mr{mod}(J(Q,W)_C)_k
&:=
\mr{mod}(J(Q,W))_k
\cap 
\mr{mod}(J(Q,W)_C),\\
\mr{mod}(J(Q',W'))^k
&:=
\{V\in \mr{mod}(J(Q',W'))\mid \Hom(V,s'_k)=0\},\\
\mr{mod}(J(Q',W')_{C'})^k
&:=
\mr{mod}(J(Q',W'))^k
\cap 
\mr{mod}(J(Q',W')_{C'}).
\end{align*}
Note that for $V\in \mr{mod}(J(Q,W))$ the condition $\Hom(s_k,V)=0$ is equivalent to the injectiveness of
\[
\sum_{s(b)=k}V_k\to \biggl(\,\bigoplus_{s(b)=k} V_{t(b)}\biggr)
\]
and for $V\in \mr{mod}(J(Q',W'))$ the condition $\Hom(V,s'_k)=0$ is equivalent to the surjectiveness of
\[
\sum_{t(a)=k}\biggl(\,\bigoplus_{t(a)=k} V_{s(a)}\biggr)\to V_k.
\]

It is shown in \cite{keller-completion}, the derived equivalence $\Phi$ induces 
\[
\mr{mod}(J(Q,W))_k
\simeq
\mr{mod}(J(Q',W'))^k.
\]
\begin{prop}\label{prop_tilting}
The derived equivalence induces 
\[
\mr{mod}(J(Q,W)_C)_k
\simeq
\mr{mod}(J(Q',W')_{C'})^k.
\]
\end{prop}
\begin{proof}
First, take $E\in \mr{mod}(J(Q,W)_C)_k$.
For a composition $ba$ through $k$, $[bc]$ vanishes on $\Phi(E)$ since $a$ vanishes on $E$.
Hence we have $\Phi(E)\in \mr{mod}(J(Q',W')_{C'})^k$.

Next, assume that $\Phi(E)\in \mr{mod}(J(Q',W')_{C'})^k$ for $E\in \mr{mod}(J(Q,W))_k$.
For any $a$ with $t(a)=k$, $a$ vanishes on $E$ since 
$\sum_{s(b)=k}[ba]$ vanished and $\sum_{s(b)=k}b$ is injective. 
Hence we have $E\in \mr{mod}(J(Q,W)_{C})_k$.
\end{proof}
\begin{rem}
This gives a generalization of a part of the results of \cite{bridgeland-stern}.
\end{rem}
\end{NB}

%\section{Motivic invariant for a QP with a cut}
%\section{Factorization property for motivic invariants}
\section{Factorization property}\label{sec_FP}
%Let $(Q,W)$ be a QP and $C$ be its cut.
%The grading $g_C$ give a $\mathbb{C}^*$-action on $\mathrm{M}(Q;\mathbf{v})$ and the action satisfies the assumption of Theorem \ref{thm_BBS} if we put 
Let $(Q,W)$ be a QP and $C$ be a cut.
%In this section, we assume that we have a nonnegative grading
%\[
%\mr{deg}\colon Q_1\to \mathbb{Z}_{\geq 0}
%g\colon Q_1\to \mathbb{Z}_{\geq 0}
%\]
%such that $W$ is homogenious of degree $d>0$.
The grading $g_C$ gives a $\mathbb{C}^*$-action on $\mathrm{M}(Q;\mathbf{v})$ and the action satisfies the assumption of Theorem \ref{thm_BBS} if we put 
\[
X=\mathrm{M}(Q;\mathbf{v}), \quad
T=\mathbb{C}^*,\quad
f=f_{W,\mathbf{v}}.
\]
Hence we have
\begin{equation}\label{eq_no_monodoromy}
%[\mathrm{crit}(f_{W,\mathbf{v}})]_{\mathrm{vir}}\in \mathcal{M}_\mathbb{C}
\mathfrak{M}_{\mathrm{vir}}(Q,W;\mathbf{v})
\in
%\mathcal{M}_\mathbb{C}[(1-\mathbb{L}^n)^{-1}:n\geq 1]
\widetilde{\mathcal{M}}_\mathbb{C}
%[(1-\mathbb{L}^n)^{-1}:n\geq 1]
\end{equation}
and 
\begin{align*}
%[\mathrm{crit}(f_{W,\mathbf{v}})]_{\mathrm{vir}}
\mathfrak{M}_{\mathrm{vir}}(Q,W;\mathbf{v})
&=
\frac{\mathbb{L}^{-\dim \mathrm{M}(Q;\mathbf{v})/2}\cdot (f_{W,\mathbf{v}}^{-1}(1)-f_{W,\mathbf{v}}^{-1}(0))}
{[G(\mathbf{v})]_\mr{vir}}\\
%\times \mathbb{L}^{\bigl(\sum_{i,j\in Q_0}Q(i,j)v_iv_j\bigr)}
%\times \mathbb{L}^{\frac{\dim \mathrm{M}(Q;\mathbf{v})}{2}}
&=
\mathbb{L}^{\chi_Q(\mathbf{v},\mathbf{v})/2}\times \frac{f_{W,\mathbf{v}}^{-1}(1)-f_{W,\mathbf{v}}^{-1}(0)}
{[G(\mathbf{v})]}
.
\end{align*}
%Note that we have
%\[
%\dim \mathrm{M}(Q;\mathbf{v})=\sum_{i,j\in Q_0}Q(i,j)v_iv_j.
%\]
%In this section, we will prove the factorization property of motivic DT invariants for any QP with cut.
We define the {\it refined DT invariant} by
\[
m_{\mr{ref}}(Q,W;\mathbf{v})
:=
\mathbb{W}\bigl(
\mathfrak{M}_{\mathrm{vir}}(Q,W;\mathbf{v})
\bigr)
\in \mathbb{Q}(q^{1/2}).
\]
%{eq_no_monodoromy}

Throughout this section, we will use simplified notations such as 
$\mathrm{M}(\mathbf{v})$ and $f_{\mathbf{v}}$
instead of 
$\mathrm{M}(Q;\mathbf{v})$ and $f_{W,\mathbf{v}}$
omitting $Q$ and $W$.

\subsection{Filtration and motivic invariants}
\begin{NB}
Let $X$, $T$, $f$ be as in Theorem \ref{thm_BBS} and
% be a regular function on a smooth variety $X$ with $T$ 
\[
0=U_0\subset U_1\subset \cdots \subset U_n =X
\]
be a filtration of $X$ by $T$-invariant open subsets. 
We put $X_\alpha:=U_\alpha\backslash U_{\alpha-1}$ and $f_\alpha:=f|_{X_\alpha}$.
Assume that $X_\alpha$, $f_\alpha$ and the $T$-action satisfies the conditions in Theorem \ref{thm_BBS} too. 
The next proposition directly follows Theorem \ref{thm_BBS} and Definition \ref{defn_vir}.
\end{NB}
The next proposition directly follows Theorem \ref{thm_BBS} and Definition \ref{defn_vir}.
\begin{prop}\label{prop_31}
Let $X$, $T$, $f$ be as in Theorem \ref{thm_BBS} and $Y$ be a smooth $T$-invariant closed subset $X$ with dimension $d$.
We assume that the $T$-action on $Y$ and $f|_Y$ satisfies the conditions in Theorem \ref{thm_BBS} as well.
Then we have
\[
[\mathrm{crit}(f)]_\mathrm{vir}
=
[\mathrm{crit}(f|_{X\backslash Y})]_\mathrm{vir}
+
\mathbb{L}^{-\frac{d}{2}}
[\mathrm{crit}(f|_{Y})]_\mathrm{vir}.
\]
\end{prop}
\begin{cor}
Let 
\[
0=U_0\subset U_1\subset \cdots \subset U_n =X
\]
be a filtration of $X$ by $T$-invariant open subsets. 
We put $X_\alpha:=U_\alpha\backslash U_{\alpha-1}$ and $f_\alpha:=f|_{X_\alpha}$.
Assume that $X_\alpha$, $f_\alpha$ and the $T$-action satisfies the conditions in Theorem \ref{thm_BBS}. Then we have
\[
[\mathrm{crit}(f)]_\mathrm{vir}
=
\sum_{\alpha=1}^n \mathbb{L}^{\frac{-\mr{dim}X+\mr{dim}X_\alpha}{2}}
[\mathrm{crit}(f_\alpha)]_\mathrm{vir}.
\]
\end{cor}

%\subsection{Factorization property}
\subsection{Filtration by HN property}\label{subsec_ks52}
In this subsection, we repeat \cite[\S 5.2]{COHA} to fix the notations.
We put 
\[
\mathbb{H}:=\{z\in \mathbb{C}\mid \mathrm{im} z >0\text{ or } z\in \mathbb{R}_{>0}\}
\]
and define the total order $\succ$ on $\mathbb{H}$ by
%define  a total order (a lexicographic order in polar coordinates) on $\mathbb{H}$ by
\[
z_1 \succ z_2 
\overset{\mathrm{def}}{\iff} 
\mathrm{Arg}(z_1) > \mathrm{Arg}(z_2)
\text{ or }
\bigl\{ \mathrm{Arg}(z_1) = \mathrm{Arg}(z_2), |z_1|>|z_2|\bigr\}.
\]
We identify the Grothendieck group $K_0(\mathrm{mod}(J(Q,W)))$ with $\mathbb{Z}^{Q_0}$ and put $N:=(\mathbb{Z}_{\geq 0})^{Q_0}$.
Let 
\[
%Z\colon K_0(\mathrm{mod}J(Q,W))\simeq\mathbb{Z}^{Q_0}\to \mathbb{C}
Z\colon \mathbb{Z}^{Q_0}\to \mathbb{C}
\]
be a central charge, that is, a group homomorphism such that 
\[
%Z\bigl((\mathbb{Z}_{\geq 0})^{Q_0}\backslash\{{0}\}\bigr)\subset\mathbb{H}.
Z(N\backslash\{{0}\})\subset\mathbb{H}.
\] 
%\[
%Z(E)\in \mathbb{H}:=\{z\in \mathbb{C}\mid \mathrm{im} z >0\text{, or } z\in \mathbb{R}_{>0}\}
%\]
%for any $0 \neq E\in \mathrm{mod}J(Q,W)$.
%Let $V\subset \mathbb{H}$ be a sector, that is, a convex subset such that $\mathbb{R}_{>0}\cdot V=V$.
%We define  a total order (a lexicographic order in polar coordinates on $\mathbb{H}$) on by
%\[
%z_1 > z_2 
%\iff 
%\mathrm{Arg}(z_1) > \mathrm{Arg}(z_2)
%\text{ or }
%\bigl( \mathrm{Arg}(z_1) = \mathrm{Arg}(z_2), |z_1|=|z_2|\bigr).
%\]
Let $\mathrm{M}_{Z\text{-ss}}(\mathbf{v})$ denote the open subset of $\mathrm{M}(\mathbf{v})$
consisting of $Z$-semistable $Q$-modules.

\begin{defn}
For $\mathbf{v}\in N$, 
we define the finite set $P_Z(\mathbf{v})$ by
%the set of collections
\[
%\mathbf{v}_\bullet=(\mathbf{v}_1,\ldots,\mathbf{v}_n)\mid n\geq 1, 
\Bigl\{
\mathbf{v}_\bullet=(\mathbf{v}_i)\in N^n\,\Big |\, n\geq 1,\ \sum \mathbf{v}_i=\mathbf{v},\ 
\mathrm{Arg}Z(\mathbf{v}_1)>\cdots >\mathrm{Arg}Z(\mathbf{v}_n)
\Bigr\}.
\]
We introduce a partial order on $P_Z(\mathbf{v})$ by 
\begin{align*}
(\mathbf{v}_1,\ldots,\mathbf{v}_n)
&\underset{{\small Z}}{<}
(\mathbf{v}'_1,\ldots,\mathbf{v}'_{n'})\\
\overset{\mathrm{def}}{\iff}
&
\mathbf{v}_1=\mathbf{v}'_1,\ldots, \mathbf{v}_{i-1}=\mathbf{v}'_{i-1} \text{ and } Z(\mathbf{v}_i)\prec Z(\mathbf{v}'_{i'})\text{ for some $i$}.
\end{align*}
\end{defn}

\begin{defn}
Let denote by
$\mathrm{M}(\mathbf{v};\mathbf{v}_\bullet)$ the subset of $\mathrm{M}(\mathbf{v})$ consisting of $Q$-modules which
%whose $\mathbb{C}$-points are objects $E\in \mathrm{M}(\mathbf{v};\mathbf{v}_\bullet)$ which
admit increasing filtrations
\[
0 = E_0\subset E_1 \subset \cdots \subset E_n=E
\]
such that 
\[
\underline{\mathrm{dim}}(E_i/E_{i-1})=\mathbf{v}_i
\]
for any $i=1,\ldots,n$.
\end{defn}

\begin{lem}
The subset $\mathrm{M}(\mathbf{v};\mathbf{v}_\bullet)\subset \mathrm{M}(\mathbf{v})$ is closed.
\end{lem}
\begin{proof}
We put
\[
\mathrm{Fl}(\mathbf{v}_\bullet)
:=
\prod_{i\in Q_0}
\mathrm{Fl}(v_{1,i},\ldots,v_{n,i})
\]
where $\mathrm{Fl}(v_{1,i},\ldots,v_{n,i})$ is the flag varieties of all flags in $V_i$ of type $(v_{1,i},\ldots,v_{n,i})$.
Note that the following subset of $\mathrm{M}(\mathbf{v})\times \mathrm{Fl}(\mathbf{v}_\bullet)$ is closed :
\[
\bigl\{(X,F)\in \mathrm{M}(\mathbf{v})\times \mathrm{Fl}(\mathbf{v}_\bullet)\,\big|\,\text{$F$ is $X$-stable}\bigr\}.
\]
Then $\mathrm{M}(\mathbf{v};\mathbf{v}_\bullet)$ is closed since it is the image of the closed set above under the projection
\[
\mathrm{M}(\mathbf{v})\times \mathrm{Fl}(\mathbf{v}_\bullet)\to 
\mathrm{M}(\mathbf{v})
\]
which is proper.
\end{proof}
\begin{defn}
We define the locally closed subset $\mathrm{M}_{Z\textup{-HN}}(\mathbf{v},\mathbf{v}_\bullet)$ of
$\mathrm{M}(\mathbf{v})$ by
%$\mathrm{M}^{\mathrm{HN}}(\mathbf{v},\mathbf{v}_\bullet)\subset \mathrm{M}(\mathbf{v})$ is smooth and its codimension is given by
\[
\mathrm{M}_{Z\textup{-HN}}(\mathbf{v},\mathbf{v}_\bullet)
:=
\mathrm{M}(\mathbf{v};\mathbf{v}_\bullet)
-
\bigcup_{\mathbf{v}'_\bullet\underset{{\small Z}}{<}\mathbf{v}_\bullet}
\mathrm{M}(\mathbf{v};\mathbf{v}'_\bullet).
\]
\end{defn}
\begin{lem}\label{lem_36}
\begin{itemize}
\item[\textup{(1)}]
A $\mathbb{C}$-point in $\mathrm{M}_{Z\textup{-HN}}(\mathbf{v},\mathbf{v}_\bullet)$ represents a $Q$-module whose HN filtration is of type $\mathbf{v}_\bullet$. 
\item[\textup{(2)}]
$\mathrm{M}_{Z\textup{-HN}}(\mathbf{v},\mathbf{v}_\bullet)$ is smooth and 
\[
\mathrm{codim}\,\mathrm{M}_{Z\textup{-HN}}(\mathbf{v},\mathbf{v}_\bullet)
=
-\sum_{a<b}\chi_Q(\mathbf{v}_a,\mathbf{v}_b).
\]
\end{itemize}
\end{lem}
\begin{proof}
(1) See \cite[pp49, Lemma 2]{COHA}.

\noindent (2) 
%For $\mathbf{v}_\bullet\in P_Z(\mathbf{v})$,  
Fix a direct sum decompositions $V_i=\oplus_{a=1}^n V_{a,i}$ with $V_{a,i}\simeq \mathbb{C}^{v_{a,i}}$. 
We define the subspace $\mathrm{M}(\mathbf{v}_\bullet)$ of $\mathrm{M}(\mathbf{v})$ by 
\[
\mathrm{M}(\mathbf{v}_\bullet):=\bigoplus_{a\geq b,\;e\in {Q_1}}
\mathrm{Hom}(V_{a,t(e)},V_{b,h(e)}).
\]
%and subgroup $G(\mathbf{v}_\bullet)$ of $G(\mathbf{v})$ by
%\[
%G(\mathbf{v}_\bullet):=
%\prod_{i\in Q_0} \{X\in \mr{GL}(V_i)\mid X(V_{a,i})\subset \oplus_{b\geq a}V_{b,i}\}.
%\]
We put
\[
\mathrm{M}_{Z\text{-ss}}(\mathbf{v}_\bullet)
:=
\pi^{-1}
\bigl(
\mathrm{M}_{Z\text{-ss}}(\mathbf{v}_1)
\times\cdots\times \mathrm{M}_{Z\text{-ss}}(\mathbf{v}_n)
\bigr)
\]
where 
\[
\pi \colon \mathrm{M}(\mathbf{v}_\bullet)\to
\mathrm{M}(\mathbf{v}_1)
\times\cdots\times 
\mathrm{M}(\mathbf{v}_n)
\]
is the natural projection.
Note that $\pi$ is a trivial vector bundle
% of rank
%\[
%\sum_{a<b,e\in Q_1}v_{a,s(e)}\cdot v_{b,t(e)}
%\]
and so $\mathrm{M}_{Z\text{-ss}}(\mathbf{v}_\bullet)$ is smooth.
\begin{NB}
We can see that
\[
\mathrm{M}_{Z\text{-ss}}(Q;\mathbf{v}_\bullet)\times_{G(\mathbf{v}_\bullet)}G(\mathbf{v})
\simeq 
\mathrm{M}_{Z\text{-HN}}(Q;\mathbf{v},\mathbf{v}_\bullet)
\]
and so $\mathrm{M}_{Z\text{-HN}}(Q;\mathbf{v},\mathbf{v}_\bullet)$ is smooth. 
\end{NB}

Let 
\[
\mr{HN}\colon 
\mathrm{M}_{Z\text{-HN}}(\mathbf{v},\mathbf{v}_\bullet)
\to 
\mr{FL}(\mathbf{v}_\bullet)
\]
be the map defined by taking the Harder-Narashimhan filtration.
Then $\mr{HN}$ is a Zariski locally trivial fibration whose fibers are isomorphic to $\mathrm{M}_{Z\text{-ss}}(Q;\mathbf{v}_\bullet)$.
So $\mathrm{M}_{Z\text{-HN}}(Q;\mathbf{v},\mathbf{v}_\bullet)$ is smooth.

The computation of the codimension is straightforward.
\end{proof}

\begin{NB}

\[
\mathrm{codim}\,\mathrm{M}^{\mathrm{HN}}(\mathbf{v},\mathbf{v}_\bullet)
=
-\sum_{a<b}\chi_Q(\mathbf{v}_a,\mathbf{v}_b).
\]

A $\mathbb{C}$-valued point of the subset
\[
\mathrm{M}_{Z\text{-ss}}(Q;\mathbf{v}_\bullet)\times_{G(\mathbf{v}_\bullet)}G(\mathbf{v})
\subset 
\mathrm{M}_{Z\text{-HN}}(Q;\mathbf{v},\mathbf{v}_\bullet)
\]
represents a $Q$-module whose HN filtration is of type $\mathbf{v}_\bullet$.
Let $\mathrm{M}_{Z\text{-ss}}(Q;\mathbf{v}_\bullet)$ denote this smooth subvariety.
\begin{defn}
\begin{itemize}
\item[(1)]
\item[(2)]
\item[(3)]
We denote by
$\mathrm{M}(\mathbf{v};\mathbf{v}_\bullet)$ the subset of $\mathrm{M}(\mathbf{v})$ whose $\mathbb{C}$-points are objects $E\in \mathrm{M}(\mathbf{v};\mathbf{v}_\bullet)$ which
admit an increasing filtration 
\[
0 = E_0\subset E_1 \subset \cdots \subset E_n
\]
such that 
\[
\underline{\mathrm{dim}}(E_i/E_{i-1})=\mathbf{v}_i
\]
for any $i=1,\ldots,n$.
\end{itemize}
\end{defn}
\begin{lem}\label{lem_HNlocus}
\begin{itemize}\textup{(}See \textup{\cite[pp49, Lemma 2 and pp51]{COHA}}.\textup{)}
\item[\textup{(1)}] The subset $\mathrm{M}(\mathbf{v};\mathbf{v}_\bullet)\subset \mathrm{M}(\mathbf{v})$ is closed.
\item[\textup{(2)}] For any $\mathbf{v}\in C$ and $\mathbf{v}_\bullet\in P(\mathbf{v})$, the set of $\mathbb{C}$-points of the locally closed subset
\[
%\mathrm{M}_{\mathbf{v},\mathbf{v}_\bullet}-\bigcup_{\mathbf{v}'_\bullet<\mathbf{v}_\bullet}\mathrm{M}_{\mathbf{v},\mathbf{v}'_\bullet}
\mathrm{M}(\mathbf{v},\mathbf{v}_\bullet)-
\bigcup_{\mathbf{v}'_\bullet<\mathbf{v}_\bullet}\mathrm{M}(\mathbf{v},\mathbf{v}'_\bullet)
\]
is the set of representations whose HN-filtrations are of type $\mathbf{v}_\bullet$. 
We denote the locally closed subset by $\mathrm{M}^{\mathrm{HN}}(\mathbf{v},\mathbf{v}_\bullet)$.
\item[\textup{(3)}] 
The locally closed subset $\mathrm{M}^{\mathrm{HN}}(\mathbf{v},\mathbf{v}_\bullet)\subset \mathrm{M}(\mathbf{v})$ is smooth and its codimension is given by
\[
\mathrm{codim}\,\mathrm{M}^{\mathrm{HN}}(\mathbf{v},\mathbf{v}_\bullet)
=
-\sum_{a<b}\chi_Q(\mathbf{v}_a,\mathbf{v}_b).
\]
\end{itemize}
\end{lem}
\end{NB}

Let $f^{Z\textup{-ss}}_{\mathbf{v}}$ (resp. $f^{Z\textup{-ss}}_{\mathbf{v}_\bullet}$, $f^{Z\textup{-HN}}_{\mathbf{v}_\bullet}$) denote the restriction of the Chern-Simons functional $f_{\mathbf{v}}=f_{W,\mathbf{v}}$ on $\mathrm{M}_{Z\textup{-ss}}(\mathbf{v})$ (resp. $\mathrm{M}_{Z\textup{-ss}}(\mathbf{v}_\bullet)$, $\mathrm{M}_{Z\textup{-HN}}(\mathbf{v},\mathbf{v}_\bullet)$).
\begin{prop}[\protect{see \cite[pp51 Theorem 5]{COHA}}]\label{prop_factorization}
Assume that the QP has a cut, then we have
\[
%\mathrm{M}^{\mathrm{HN}}(\mathbf{v},\mathbf{v}_\bullet)
\frac{\bigl[\mr{crit}\bigl(f^{Z\textup{-HN}}_{\mathbf{v}_\bullet}\bigr)\bigr]_{\mr{vir}}}{[\mr{G}(\mathbf{v})]_\mr{vir}}
=
\mathbb{L}^{-\sum_{a>b}\chi_Q(\mathbf{v}_a,\mathbf{v}_b)}
\times \prod_{a}
\frac{\bigl[\mr{crit}\bigl(f^{Z\textup{-ss}}_{\mathbf{v}_a}\bigr)\bigr]_{\mr{vir}}}{[\mr{G}(\mathbf{v}_a)]_\mr{vir}}.
\]
\end{prop}
\begin{proof}
The claim is a consequence of the following two identities, which is obtained from the descriptions of in the proof of Lemma \ref{lem_36} :
\begin{align*}
\bigl[\mr{crit}\bigl(f^{Z\textup{-HN}}_{\mathbf{v}_\bullet}\bigr)\bigr]_{\mr{vir}}
&=
\bigl[\mr{crit}\bigl(f^{Z\textup{-ss}}_{\mathbf{v}_\bullet}\bigr)\bigr]_{\mr{vir}}
\times
[\mr{FL}(\mathbf{v}_\bullet)]_\mr{vir}\\
&= \bigl[\mr{crit}\bigl(f^{Z\textup{-ss}}_{\mathbf{v}_\bullet}\bigr)\bigr]_{\mr{vir}}
\times
%\frac{[\mr{G}(\mathbf{v})]_\mr{vir}}{\prod_a [\mr{G}(\mathbf{v_a})]_\mr{vir}\times \prod_{i}\prod_{a>b}[\mathbb{L}^{v_{a,i}\cdot v_{b,i}}]_\mr{vir}}
\frac{[\mr{G}(\mathbf{v})]_\mr{vir}}{\prod_a [\mr{G}(\mathbf{v_a})]_\mr{vir}\times \bigl[\mathbb{L}^{\sum_{i,a>b}v_{a,i}\cdot v_{b,i}}\bigr]_\mr{vir}}
\end{align*}
and
\begin{align*}
\bigl[\mr{crit}\bigl(f^{Z\textup{-ss}}_{\mathbf{v}_\bullet}\bigr)\bigr]_{\mr{vir}}
&=
\Bigl[\mathbb{L}^{\sum_{i,a>b} Q_{ij}v_{a,i}\cdot v_{b,i}}\Bigr]_\mr{vir}\times \prod_a 
\bigl[\mr{crit}\bigl(f^{Z\textup{-ss}}_{\mathbf{v}_a}\bigr)\bigr]_{\mr{vir}}.
\end{align*}
For the second identity, we use the motivic Thom-Sebastiani theorem (Theorem \ref{thm_thom_sebastiani}).
%First, using 
%\begin{itemize}
%\item Lemma \ref{lem_13}, Lemma \ref{lem_12} and 
%\item the description of $\mathrm{M}_{Z\textup{-HN}}(\mathbf{v},\mathbf{v}_\bullet)$ in the proof of Lemma \ref{lem_36},
%\end{itemize}
%we can describe the left hand side in terms of the product of semi-stable loci 
%\[
%\mathrm{M}_{Z\text{-ss}}(Q;\mathbf{v}_1)
%\times\cdots\times \mathrm{M}_{Z\text{-ss}}(Q;\mathbf{v}_n).
%\]
%Then the claim follows by the motivic Thom-Sebastiani theorem (Theorem \ref{thm_thom_sebastiani}).
\end{proof}
Let $\langle\bullet,\bullet\rangle \colon \mathbb{Z}^{Q_0} \times \mathbb{Z}^{Q_0} \to \mathbb{Z}$ be the skewsymmetric bilinear form given by 
\[
\langle\mathbf{v},\mathbf{v}'\rangle:=
\chi_Q(\mathbf{v},\mathbf{v}')-\chi_Q(\mathbf{v}',\mathbf{v}).
\]
Combining the results in this subsection, we get the following theorem :
\begin{thm}\label{thm_39}
Assume that the QP has a cut, then we have
\[
\mathfrak{M}_{\mr{vir}}(Q,W,\mathbf{v})
:=
\sum_{\mathbf{v}_\bullet\in P_Z(\mathbf{v})}
\biggl(\mathbb{L}^{\frac{1}{2}\sum_{a<b} \langle\mathbf{v}_a,\mathbf{v}_b\rangle}
\times
\prod_a \mathfrak{M}_{\mr{vir}}(Q,W,\mathbf{v}_a)\biggr)
\]
\end{thm}
%\begin{proof}
%\end{proof}

\subsection{Factorization property}
We assume that the QP has a cut.
%In order to describe the factorization property, we introduce the {\it quantum torus} :
\begin{defn}\label{defn_mt}
%The quantum torus associated to $Q$ is
The {\it motivic torus} associated to $Q$ is
\[
\hat{\mathcal{T}}_Q:=
%\prod_{\mathbf{v}\in (\mathbb{Z}_{\geq 0})^{Q_0}} \mathcal{M}^{\hat{\mu}}_\mathbb{C}[(1-\mathbb{L}^n)^{-1}:n\geq 1]\cdot y_\mathbf{v}
%\prod_{\mathbf{v}\in (\mathbb{Z}_{\geq 0})^{Q_0}} \widetilde{\mathcal{M}}^{\hat{\mu}}_\mathbb{C}\cdot y_\mathbf{v}
\prod_{\mathbf{v}\in N} \widetilde{\mathcal{M}}_\mathbb{C}\cdot y_\mathbf{v}
\]
where $y_\mathbf{v}$'s are formal variables which satisfy the relation
\[
y_{\mathbf{v}_1}\cdot y_{\mathbf{v}_2}
=
\mathbb{L}^{\frac{\langle\mathbf{v},\mathbf{v}'\rangle}{2}}y_{\mathbf{v}_1+\mathbf{v}_2}.
\]
\end{defn}
\begin{defn}
We define the generating series of the motivic Donaldson-Thomas invariants of $(Q,W)$ by
\[
\mca{A}=\mca{A}_{Q,W}:=
%\sum_{\mathbf{v}\in (\mathbb{Z}_{\geq 0})^{Q_0}}
1+\sum_{\mathbf{v}\in N}
%[\mr{crit}(f_{\mathbf{v}})]_{\mr{vir}}\cdot y_\mathbf{v}
\mathfrak{M}_{\mr{vir}}(Q,W,\mathbf{v})\cdot y_\mathbf{v}
\in 
\hat{\mathcal{T}}_Q.
\]
\end{defn}
\begin{defn}
Let $l\subset \mathbb{H}$ be a ray and $Z$ be a central charge.  
We put
\[
\mca{A}^{Z,l}:=
1+\sum_{Z(\mathbf{v})\in l}
%[\mr{crit}(f^{Z\textup{-ss}}_{\mathbf{v}})]_{\mr{vir}}\cdot y_\mathbf{v}
\frac{[\mr{crit}(f^{Z\textup{-ss}}_{\mathbf{v}})]_{\mr{vir}}}{[G(\mathbf{v})]}\cdot y_\mathbf{v}
\in 
\hat{\mathcal{T}}_Q.
\]
\end{defn}
%Proposition \ref{prop_factorization} implies the following {\it factorization formula} : 
Theorem \ref{thm_39} implies the following {\it factorization formula} : 
\begin{thm}\label{thm_FP}
Assume that the QP has a cut, then we have
%({\bf foctorization formula})
\[
\mca{A}=\prod_{l}^{\curvearrowright}\mca{A}^{Z,l}
\]
where the product is taken in the clockwise order over all rays.
\end{thm}
\begin{defn}\label{defn_qt}
The {\it quantum torus associated} to $Q$ is
\[
\hat{{T}}_Q:=
\prod_{\mathbf{v}\in N} \mathbb{Q}(q^{1/2})\cdot \mr{y}_\mathbf{v}
\]
where $\mr{y}_\mathbf{v}$'s are formal variables which satisfy the relation
\[
\mr{y}_{\mathbf{v}_1}\cdot \mr{y}_{\mathbf{v}_2}
=
q^{\frac{\langle\mathbf{v},\mathbf{v}'\rangle}{2}}\mathrm{y}_{\mathbf{v}_1+\mathbf{v}_2}.
%q^{-\chi_Q(\mathbf{v}_1,\mathbf{v}_2)}\mr{y}_{\mathbf{v}_1+\mathbf{v}_2}.
\]
\end{defn}
\begin{defn}\label{defn_rdt}
We define the generating series of the refined Donaldson-Thomas invariants of $(Q,W)$ by
\[
{A}={A}_{Q,W}:=
m_{\mr{ref}}(Q,W,\mathbf{v})\cdot \mr{y}_\mathbf{v}
\in 
\hat{{T}}_Q.
\]
\end{defn}
In the same way, we define ${A}^{Z,l}$. 
Since $\mathbb{W}$ is a ring homomorphism, we get the following
factorization formula for refined DT invariants : 
\begin{cor}\label{cor_FP}
Assume that the QP has a cut, then we have
\[
{A}=\prod_{l}^{\curvearrowright}{A}^{Z,l}
\]
where the product is taken in the clockwise order over all rays.
\end{cor}

%\subsection{Application : quantum dilogarithm identity}

\section{Wall-crossing formula}\label{sec_WC}
\subsection{Motives for $J(Q,W)_C$}
We put
%\begin{align*}
%d&:=\dim \mathrm{M}(Q;\mathbf{v})-\mathrm{M}(Q_C;\mathbf{v})\\
%&\sum_{c\in C}v_{t(c)}v_{h(c)}.
%\end{align*}
\begin{align*}
\chi_C(\mathbf{v},\mathbf{v'})&:=\sum_{c\in C}v_{t(c)}v'_{h(c)},\\
\chi_{Q_C}(\mathbf{v},\mathbf{v'})&:=
\chi_Q(\mathbf{v},\mathbf{v'})-\chi_C(\mathbf{v},\mathbf{v'}).
%=\sum_{c\in Q_C}v_{t(c)}v'_{h(c)}.
\end{align*}
%\[
%d:=\dim \mathrm{M}(Q;\mathbf{v})-\dim \mathrm{M}(Q_C;\mathbf{v})=\sum_{c\in C}v_{t(c)}v_{h(c)}.
%\]
%We define 
%\[
%\mathrm{M}(Q,W,C;\mathbf{v}) \subset \mathrm{M}(Q_C;\mathbf{v})
%\]
%by the closed subset of $J(Q,W)_C$-modules and put
%We define 
%\[
%\mathfrak{M}(Q,W,C;\mathbf{v}):=
%\frac{\mathbb{L}^{-\frac{\dim \mr{M}(Q_C;\mathbf{v})}{2}}\cdot [\mathrm{M}(Q,W,C;\mathbf{v})]}{[G(\mathbf{v})]_\mr{vir}}\in \widetilde{\mathcal{M}}_\mathbb{C}
%\mathcal{M}_\mathbb{C}[(1-\mathbb{L}^n)^{-1}:n\geq 1]
%\frac{[\mathrm{M}(Q,W,C;\mathbf{v})]}{[G(\mathbf{v})]}\in \widetilde{\mathcal{M}}_\mathbb{C}
%\]
We put
\[
d:=\dim \mr{M}(Q;\mathbf{v}) - \dim \mr{M}(Q_C;\mathbf{v})
=\chi_C(\mathbf{v},\mathbf{v}).
\]
Let $\mathrm{M}(Q,W,C;\mathbf{v})$ be the subset of $\mathrm{M}(Q_C;\mathbf{v})$ consisting of $J(Q,W)_C$-modules.
The following theorem is a generalization of \cite[Equation (2.4)]{behrend_bryan_szendroi} and \cite[Theorem 7.3]{hua}.
\begin{thm}\label{thm_cut_id}
%\[
%\mathfrak{M}_{\mr{vir}}(Q,W;\mathbf{v})
%=
%\mathfrak{M}(Q,W,C;\mathbf{v}).
%\]
\[
[\varphi_{f_{W,\mathbf{v}}}]
=
\mathbb{L}^d\cdot [\mathrm{M}(Q,W,C;\mathbf{v})].
\]
\end{thm}
\begin{proof}
Note that we have
\begin{align}
[\varphi_{f_{W,\mathbf{v}}}]
&=
f_{W,\mathbf{v}}^{-1}(1)
-
f_{W,\mathbf{v}}^{-1}(0)\notag\\
&=
\frac{[\mr{M}(Q;\mathbf{v})]-f_{W,\mathbf{v}}^{-1}(0)}{\mathbb{L}-1}
-
f_{W,\mathbf{v}}^{-1}(0)\notag\\
&=
\frac{[\mr{M}(Q;\mathbf{v})]-\mathbb{L}\cdot f_{W,\mathbf{v}}^{-1}(0)}{\mathbb{L}-1}.\label{eq_1}
\end{align}
Let $\pi\colon \mr{M}(Q;\mathbf{v})\to \mr{M}(Q_C;\mathbf{v})$ be the natural projection. 
This is a trivial vector bundle of rank $d$.
% with fibers isomorphic to $\mathbb{C}^{d}$. 
Since we have
\[
W=\sum_{e\in C}\partial_eW, 
\]
the restriction of $f_{W,\mathbf{v}}$ to the fiber $\pi^{-1}(x)$ is zero if $x\in \mathrm{M}(Q,W,C;\mathbf{v})$ and is a non-zero linear function if $x\notin \mathrm{M}(Q,W,C;\mathbf{v})$.
Hence we have
\begin{equation}\label{eq_2}
f_{W,\mathbf{v}}^{-1}(0)=
\mathbb{L}^d\cdot [\mathrm{M}(Q,W,C;\mathbf{v})] 
+
\mathbb{L}^{d-1}\bigl([\mathrm{M}(Q_C;\mathbf{v})]-[\mathrm{M}(Q,W,C;\mathbf{v})]\bigr).
\end{equation}
Substitute \eqref{eq_2} to \eqref{eq_1}, then the claim follows.
\end{proof}
\begin{rem}
For the cohomological Hall algebra, a similar statement is proved in \cite[Proposition 6]{COHA}.
\end{rem}

\subsection{Mutation and bilinear forms}
Let $(Q',W',C')$ be the new QP with the cut given by the mutation at a strict source $k$ of $(Q,W,C)$.

We identify the Grothendieck group $K_0(\mathrm{mod}(J(Q,W)))$ with $\mathbb{Z}^{Q_0}$ as before. 
We define 
\[
\phi_k\colon 
\mathbb{Z}^{Q_0} \overset{\sim}{\longrightarrow} \mathbb{Z}^{Q'_0}
\]
by
\[
\phi_k([s_i])=
\begin{cases}
[s'_i] & i\neq k,\\
-[s'_k]+\sum_{t(b)=k}[s'_{h(b)}] & i=k.
\end{cases}
\]
Then we can verify the following :
\begin{lem}\label{lem_43}
\begin{align*}
\chi_Q(\mathbf{v},\mathbf{w})
&=
\chi_{Q'}(\phi_k(\mathbf{v}),\phi_k(\mathbf{w})),\\
\chi_C(\mathbf{v},\mathbf{w})
&=
\chi_{C'}(\phi_k(\mathbf{v}),\phi_k(\mathbf{w})),\\
\chi_{Q_C}(\mathbf{v},\mathbf{w})
&=
\chi_{Q_{C'}}(\phi_k(\mathbf{v}),\phi_k(\mathbf{w})).
\end{align*}
\end{lem}
\begin{rem}
The bilinear form $\chi_Q$ is the Euler form of the derived category of the Ginzburg's dg algebra. 
The map $\phi_k$ is induced from the Keller-Yang's derived equivalence (\cite{dong-keller,keller-completion}). 
This is the origin of the first equation. 

If the Ginzburg's dg algebra is concentrated on degree $0$, then the bilinear form $\chi_{Q_C}$ is the Euler form of the derived category of $J(Q,W)_C$. In such a cases, we have the derived equivalence between $J(Q,W)_C$ and $J(Q',W')_{C'}$ which induces $\phi_k$. This is the origin of the third equation.
\end{rem}

%\subsection{Mutated CS functional and motivic invariants}
\subsection{Mutation and a motivic identity}
For a strict source $k$, we define open subsets 
$\mathrm{M}(Q;\mathbf{v})_k\subset \mathrm{M}(Q;\mathbf{v})$ and
$\mathrm{M}(Q_C;\mathbf{v})_k\subset \mathrm{M}(Q_C;\mathbf{v})$ by 
\begin{align*}
\mathrm{M}(Q;\mathbf{v})_k&:=\{ V\in \mathrm{M}(Q;\mathbf{v})\mid \Hom(s_k,V)=0 \},\\
\mathrm{M}(Q_C;\mathbf{v})_k&:=\{ V\in \mathrm{M}(Q_C;\mathbf{v})\mid \Hom(s_k,V)=0 \}
\end{align*}
and put
\[
\mathrm{M}(Q,W,C;\mathbf{v})_k
:=
\mathrm{M}(Q_C;\mathbf{v})_k
\cap 
\mathrm{M}(Q,W,C;\mathbf{v}).
\]
Note that we have $\pi^{-1}(\mathrm{M}(Q_C;\mathbf{v})_k)=\mathrm{M}(Q;\mathbf{v})_k$. 
%We put
%\begin{align*}
%\[
%\mathfrak{M}_{\mr{vir}}(Q,W;\mathbf{v})_k
%:=
%\frac{\bigl[\mathrm{crit}\bigl(f_{W,\mathbf{v}}|_{\mathrm{M}(Q;\mathbf{v})_k}\bigr)\bigr]_\mathrm{vir}}{[G(\mathbf{v})]_\mathrm{vir}}.
%\mathfrak{M}(Q,W,C;\mathbf{v})_k
%&:=
%\mathbb{L}^{\dim \mr{M}(Q_C;\mathbf{v})_k}\cdot \frac{\bigl[\mathrm{M}\bigl(Q,W,C;\mathbf{v}\bigr)_k\bigr]}{[G(\mathbf{v})]}
%\mathbb{L}^{\dim \mr{M}(Q_C;\mathbf{v})_k}\cdot \frac{[\mathrm{M}(Q,W,C;\mathbf{v})_k]}{[G(\mathbf{v})]}
%\end{align*}
%\]
We put
\[
f_{W,\mathbf{v},k}
:=
f_{W,\mathbf{v}}|_{M(Q,W,C,\mathbf{v})_k}.
\]
We can prove the following in the same way as Theorem \ref{thm_cut_id} :
\begin{prop}\label{prop_43}
\[
[\varphi_{f_{W,\mathbf{v},k}}]
=
%\mathbb{L}^{\chi_C(\mathbf{v},\mathbf{v})}\cdot [\mathrm{M}(Q,W,C;\mathbf{v})_k].
\mathbb{L}^{d}\cdot [\mathrm{M}(Q,W,C;\mathbf{v})_k].
\]
\end{prop}
We use the upper subscript $M(\cdots){}^k$ for the ones with the condition $\Hom(-,s_k')=0$.
\begin{prop}\label{prop_44}
\begin{equation}
\frac{[\mr{M}(Q,W,C;\mathbf{v})_k]}{[G(\mathbf{v})]}
=
\frac{[\mr{M}(Q',W',C';\mathbf{v}')^k]}{[G(\mathbf{v}')]}
\end{equation}
where $\mathbf{v}'=\phi_k(\mathbf{v})$.
\end{prop}
\begin{proof}
This is a consequence of Proposition \ref{prop_tilting} and Lemma \ref{lem_comporison}. 
\footnote{In fact, the equivalence in Corollary \ref{prop_tilting} gives only the bijection of $\mathbb{C}$-valued points. But all the arguments in \S \ref{subsec_appendix} can be applied for families of representations and we can see the isomorphism of the moduli stacks.}
\end{proof}
Combining this with Proposition \ref{prop_43} and Lemma \ref{lem_43}, we get the following identity of the virtual motives of the same moduli stack with different Chern-Simons functionals :
\begin{thm}\label{thm_45}
\[
\mathfrak{M}_{\mr{vir}}(Q,W;\mathbf{v})_k
=
\mathfrak{M}_{\mr{vir}}(Q',W';\mathbf{v}')^k.
\]
\end{thm}

%\section{Application : wall-crossing phenomena}

%\subsection{}

\subsection{Wall-crossing formula for motivic DT invariants}
In this subsection, we will work over $\mathbb{Q}$ and use the notations as $\hat{\mathcal{T}}^{\mathbb{Q}}_{Q}:=\hat{\mathcal{T}}_{Q}\otimes \mathbb{Q}$.

We put $N_{Q,Q'}:=\phi_k^{-1}(N)\cap N$ and
\[
\hat{\mathcal{T}}^{\mathbb{Q}}_{Q,Q'}:=
\prod_{\mathbf{v}\in N_{Q,Q'}} \widetilde{\mathcal{M}}_\mathbb{C}\cdot y_\mathbf{v}
\subset \hat{\mathcal{T}}^{\mathbb{Q}}_{Q}.
\]
Note that $\hat{\mathcal{T}}^{\mathbb{Q}}_{Q,Q'}$ is also a subalgebra of $\hat{\mathcal{T}}^{\mathbb{Q}}_{Q'}$. We put
\begin{align*}
\mca{A}_{Q,W,k}:=&\,
1+\sum_{\mathbf{v}\in N_{Q,Q'}}
%[\mr{crit}(f_{\mathbf{v},k})]_{\mr{vir}}\cdot y_\mathbf{v}
\mathfrak{M}_{\mr{vir}}(Q,W;\mathbf{v})_k\cdot y_\mathbf{v}\\
=&\,
1+\sum_{\mathbf{v}\in N_{Q,Q'}}
\mathfrak{M}_{\mr{vir}}(Q',W';\mathbf{v}')^k\cdot y_\mathbf{v}\quad (\text{Theorem \ref{thm_45}})\\
\in&\  
\hat{\mathcal{T}}^{\mathbb{Q}}_{Q,Q'}.
\end{align*}
%where $f_{\mathbf{v},k}$ is the restriction of $f_{\mathbf{v}}$ on $\mr{M}(Q;\mathbf{v})_k$
and
\begin{align*}
\mathbb{E}(s_k)
&:=
\sum_{n\geq 0}\frac{[\mr{pt}]}{[\mr{GL}_n]_\mr{vir}}\cdot {y}_{\,[\,(s_k)^{\oplus n}]}\in \mca{T}_Q^{\mathbb{Q}},\\
\mathbb{E}(s'_k)
&:=
\sum_{n\geq 0}\frac{[\mr{pt}]}{[\mr{GL}_n]_\mr{vir}}\cdot {y}_{\,[\,(s_k')^{\oplus n}]}\in \mca{T}^{\mathbb{Q}}_{Q'}
\end{align*}
(see Example \ref{ex_dilog}).
In the same way as Theorem \ref{thm_FP}, we can see the following factorizations :
\begin{align*}
\mca{A}_{Q,W}&=\mca{A}_{Q,W,k}\times \mathbb{E}(s_k) \in \hat{\mathcal{T}}^{\mathbb{Q}}_{Q},\\
\mca{A}_{Q',W'}&=\mathbb{E}(s'_k) \times \mca{A}_{Q,W,k} \in \hat{\mathcal{T}}_{Q'}^{\mathbb{Q}}.
\end{align*}
Now we get the following wall-crossing formula for the motivic DT invariants :
\begin{thm}\label{thm_WC}
We have
\[
\mca{A}_{Q,W}\times \mathbb{E}(s_k)^{-1},\quad
\mathbb{E}(s'_k)^{-1} \times \mca{A}_{Q',W'}\in \hat{\mathcal{T}}_{Q,Q'}^{\mathbb{Q}}
\]
and they coincide.
\end{thm}

\subsection{Refined DT invariants}
\begin{NB}
We put
\[
\hat{{T}}_{Q}:=
\prod_{\mathbf{v}\in N} \mathbb{Q}(q)\cdot y_\mathbf{v}.
\]
and define the generating function of {\it refined Donaldson-Thomas invariants} by
\[
A_{Q,W}:=1+\sum_{\mathbf{v}\in N}
W\bigl(\mathfrak{M}_{\mr{vir}}(Q,W,\mathbf{v})\bigr)\cdot y_\mathbf{v}
\in \hat{{T}}_{Q}.
\]
\end{NB}
We define $\hat{T}_{Q,Q'}$ in the same way.

Taking the weight polynomial $\mathbb{W}$ of the equation in Theorem \ref{thm_WC}, we get the following formula which describes the relation between the refined DT invariants of $(Q,W)$ and $(Q',W')$ in terms of the quantum dilogarithm :
\begin{thm}
We have
\[
{A}_{Q,W}\times \mathbb{E}_q(y_k)^{-1},\quad
\mathbb{E}_q(y_k^{-1})^{-1} \times {A}_{Q',W'}\in \hat{{T}}_{Q,Q'}
\]
and they coincide.
\end{thm}

\subsection{Appendix : reminders on \cite{quiver-with-potentials}}\label{subsec_appendix}
%By Proposition \ref{prop_tilting}, we have
%This is a consequence of Proposition \ref{prop_DWZ}. 
%Let $\widetilde{\mu}_kQ$ be the quiver 
For a QP $(Q,W)$ and a vertex $k$, we associate a new QP $\widetilde{\mu}_k(Q,W)=(\widetilde{Q},\widetilde{W})$ as follows. 
We put $\widetilde{Q}_0=Q_0$ and $\widetilde{Q}_1$ is the union of 
\begin{itemize}
\item all the arrows $c\in Q_1$ not incident to $k$, 
\item a ``composite'' arrow $[ba]$ from $t(a)$ to $h(b)$ for each $a$ and $b$ with $h(a)=t(b)=k$, and
\item an opposite arrow $a^*$ (resp. $b^*$) for each incoming arrow $a$ (resp. outgoing arrow $b$) at $k$.
\end{itemize}
The new potential is given by 
\[
\widetilde{W}:=[W]+\Delta
\]
where 
\[
\Delta:=\sum_{a,b\in Q_1; h(a)=t(b)=k}[ba]a^*b^*
\]
and $[W]$ is obtained by substituting $[ba]$ for each factor $ba$ occurring in the expansion of $W$.
What we have been assuming is that there is an automorphism $\psi$ of $\mathbb{C}\widetilde{Q}$ such that we have a decomposition
\begin{equation}\label{eq_decomp}
(\psi(\widetilde{Q}),\psi(\widetilde{W}))
\simeq 
(\psi(\widetilde{Q})_{\mr{red}},\psi(\widetilde{W})_{\mr{red}})
\oplus
(\psi(\widetilde{Q})_{\mr{triv}},\psi(\widetilde{W})_{\mr{triv}})
\end{equation}
with $\psi(\widetilde{Q})_{\mr{red}}=\mu_k(Q)$. 
We put $\mu_k(W):=\psi(\widetilde{W})_{\mr{red}}$.
The reader may refer \cite[\S 5]{quiver-with-potentials} for the details\footnote{In \cite{quiver-with-potentials}, it is shown that for a generic $W$ we always have such an automorphism of the completion of $\mathbb{C}\widetilde{Q}$. Here we assume that have an automorphism of $\mathbb{C}\widetilde{Q}$, otherwise the mutation of the potential can be infinite. }. 

We fix the automorphism $\psi$ and the decomposition \eqref{eq_decomp}.
Then a $J_{\widetilde{Q},\widetilde{W}}$-module is canonically identified with a $J_{\mu_k(Q,W)}$-module.

Take $V\in \mr{M}(Q;\mathbf{v})_k$. 
We define $\widetilde{V}:=\oplus \widetilde{V}_i$ by $\widetilde{V}_i:=V_i$ for $i\neq k$ and
\[
\widetilde{V}_k
:=
\mr{coker}
\biggl(\ 
\sum_{t(b)=k}b
\colon 
V_{k}\to \bigoplus_{t(b)=k}V_{h(b)}
\,\biggr).
\]
Note that the sum of the maps above is injective.
We define the action of $\mathbb{C}\widetilde{Q}$ on $\widetilde{V}$ as follows:
\begin{itemize}
\item for an arrow $c\in Q_1$ not incident to $k$, we associate
\[
%c\colon \widetilde{V}_{t(c)}={V}_{t(c)} \to {V}_{h(c)}=\widetilde{V}_{h(c)},
\widetilde{V}_{t(c)}={V}_{t(c)} \overset{c\,\,}{\longrightarrow} {V}_{h(c)}=\widetilde{V}_{h(c)},
\]
\item for a ``composite'' arrow $[ba]$, we associate the composition
\[
%[ba]= b\circ a \colon 
\widetilde{V}_{t([ba])}={V}_{t(a)} \overset{b\circ a\,\,}{\longrightarrow} {V}_{h(b)}=\widetilde{V}_{h([ba])}, 
\]
\item for an opposite arrow $a^*$ of an incoming arrow $a$ at $k$, we associate the map induced by
\[
\sum_{t(b)=k} \partial_{[ba]}W\colon 
%\widetilde{V}_{t(a^*)}={V}_{h(a)} 
%\overset{\sum \partial_{[ba]}W\,\,}{\longrightarrow} 
\bigoplus_{t(b)=k}V_{h(b)}
\longrightarrow
V_{t(a)}=\widetilde{V}_{h(a^*)}, 
\]
\item for an opposite arrow $b^*$ of an outgoing arrow $b$ at $k$, we associate
\[
\widetilde{V}_{t(b^*)}={V}_{h(b)}
\hookrightarrow 
\bigoplus_{t(b')=k}V_{h(b')}
\twoheadrightarrow 
\widetilde{V}_k
=
\widetilde{V}_{h(b^*)}.
\]
\end{itemize}
Then $\widetilde{V}$ belongs to $\mr{M}(\widetilde{Q},\mathbf{v}')^k$ where $\mathbf{v}'=\phi_k(\mathbf{v})$. 
We can verify that if $V\in \mr{M}({Q},W, \mathbf{v})_k$ then $\widetilde{V}\in \mr{M}(\widetilde{Q},\widetilde{W}, \mathbf{v}')^k$.
%We can verify this satisfies the relations to be a $J_{\widetilde{Q},\widetilde{W}}$-module.

Take $U\in \mr{M}(\widetilde{Q};\mathbf{u})^k$. 
%We define $\widetilde{U}:=\oplus \widetilde{U}_i$ by $\widetilde{U}_i:=U_i$ for $i\neq k$ and
We define $\widehat{U}:=\oplus \widehat{U}_i$ by $\widehat{U}_i:=U_i$ for $i\neq k$ and
\[
\widehat{U}_k
:=
\mr{ker}
\biggl(\ 
\sum_{h(b^*)=k}b^*
\colon 
\bigoplus_{h(b^*)=k}U_{t(b^*)}
\to 
U_k
\,\biggr).
\]
Note that the sum of the maps above is surjective.
We define the action of $\mathbb{C}\widetilde{\widetilde{Q}}$ on $\widehat{U}$ in a similar way.
If $U\in \mr{M}(\widetilde{Q},\widetilde{W}, {\mathbf{u}})^k$ then $\widetilde{U}\in \mr{M}(\widetilde{\widetilde{Q}},\widetilde{\widetilde{W}}, \phi_k^{-1}(\mathbf{u}))_k$.

In the proof of \cite[Theorem 5.7]{quiver-with-potentials}, an explicit equivalence between $(Q,W)$ and $(\widetilde{\widetilde{Q}},\widetilde{\widetilde{W}})$ is given.  
This induces an identification of elements in $\mr{M}(\widetilde{\widetilde{Q}},\widetilde{\widetilde{W}}, \mathbf{u})_k$ with ones in $\mr{M}({Q},{W}, \mathbf{u})_k$

We can verify the composition of these there functors is identity. 
Combined with the identification of $J_{\widetilde{Q},\widetilde{W}}$-modules and $J_{\mu_k(Q,W)}$-modules we get the following equivalence:
%\footnote{The argument so far looks set theoritical, }
\begin{prop}
\begin{equation}\label{eq_tilting}
\Phi\colon \mr{mod}(J(Q,W))_k
\overset{\sim}{\longrightarrow}
\mr{mod}(J(Q',W'))^k.
\end{equation}
\end{prop}
\begin{rem}
%The equivalence between these two categories was shown in \textup{\cite{dong-keller}}.
The derived equivalence given by \textup{\cite{dong-keller}} induces the equivalence of the two categories above.
Here we use \cite{quiver-with-potentials}'s construction since we need the explicit description of $\Phi$ to prove the following proposition.
\end{rem}
\begin{prop}\label{prop_tilting}
The equivalence \eqref{eq_tilting} induces 
\begin{equation*}
\mr{mod}(J(Q,W)_C)_k
\simeq
\mr{mod}(J(Q',W')_{C'})^k.
\end{equation*}
\end{prop}
\begin{proof}
First, take $V\in \mr{mod}(J(Q,W)_C)_k$.
For each composite arrow $[ba]$, the map $[ba]$ vanishes on $\Phi(V)$ since the map $a$ vanishes on $V$.
Hence we have $\Phi(V)\in \mr{mod}(J(Q',W')_{C'})^k$.

Next, assume that $\Phi(V)\in \mr{mod}(J(Q',W')_{C'})^k$ for $V\in \mr{mod}(J(Q,W))_k$.
For any $a$ with $h(a)=k$, $a$ vanishes on $V$ since 
$\sum_{t(b)=k}[ba]$ vanished and $\sum_{t(b)=k}b$ is injective. 
Hence we have $V\in \mr{mod}(J(Q,W)_{C})_k$.
\end{proof}
\begin{rem}
This gives a generalization of a part of the results of \cite{bridgeland-stern}.
\end{rem}

\begin{NB}
Let $\Gamma_{Q,W}$ be the Ginzburg dg algebra and 
$\mathcal{D}_{Q,W}$ be the derived category of dg modules with finite dimensional cohomologies over $\Gamma_{Q,W}$.
Let 
\[
\Phi=\Phi_k \colon \mathcal{D}_{Q,W} \overset{\sim}{\longrightarrow} \mathcal{D}_{Q',W'}
\]
be the derived equivalence (\cite{dong-keller,keller-completion}) which induces the isomorphism $\phi=\phi_k$ of the Grothendieck group given by
\[
\phi([s_i])=
\begin{cases}
[s'_i]+Q(i,k)[s'_k] & i\neq k.\\
-[s'_k] & i=k.
\end{cases}
\]

Assume that $(Q,W)$ is Calabi-Yau, that is, $\Gamma_{Q,W}$ is concentrated on degree $0$. 
In such a case 
\begin{itemize}
\item[(1)] $\mathcal{D}_{Q,W}\simeq D^b(\mr{mod}J(Q,W))$, 
\item[(2)] $\mu_k(Q,W)$ is also Calabi-Yau (\cite{keller-completion}).
\end{itemize} 
Let $\Phi$ be the derived equivalence between $D^b(\mr{mod}J(Q,W))$ and $D^b(\mr{mod}J(\mu_kQ,\mu_kW))$. For any $V\in D^b(\mr{mod}J(Q,W))$, we have 
\[
H^j(\Phi(V))_i=0\ (j\neq 0),\quad H^0(\Phi(V))_i=V_i
\]
for $i\neq k$ and the map 
\[
[ba] \colon H^0(\Phi(V))_{s(a)} \to H^0(\Phi(V))_{t(b)}
\]
is given by the composition $b\circ a$.
%\begin{prop}
%\end{prop}
We define the following full subcategories 
\begin{align*}
\mr{mod}(J(Q,W))_k
&:=
\{V\in \mr{mod}(J(Q,W))\mid \Hom(s_k,V)=0\},\\
\mr{mod}(J(Q,W)_C)_k
&:=
\mr{mod}(J(Q,W))_k
\cap 
\mr{mod}(J(Q,W)_C),\\
\mr{mod}(J(Q',W'))^k
&:=
\{V\in \mr{mod}(J(Q',W'))\mid \Hom(V,s'_k)=0\},\\
\mr{mod}(J(Q',W')_{C'})^k
&:=
\mr{mod}(J(Q',W'))^k
\cap 
\mr{mod}(J(Q',W')_{C'}).
\end{align*}
Note that for $V\in \mr{mod}(J(Q,W))$ the condition $\Hom(s_k,V)=0$ is equivalent to the injectiveness of
\[
\sum_{s(b)=k}V_k\to \biggl(\,\bigoplus_{s(b)=k} V_{t(b)}\biggr)
\]
and for $V\in \mr{mod}(J(Q',W'))$ the condition $\Hom(V,s'_k)=0$ is equivalent to the surjectiveness of
\[
\sum_{t(a)=k}\biggl(\,\bigoplus_{t(a)=k} V_{s(a)}\biggr)\to V_k.
\]

It is shown in \cite{keller-completion}, the derived equivalence $\Phi$ induces 
\[
\mr{mod}(J(Q,W))_k
\simeq
\mr{mod}(J(Q',W'))^k.
\]
\begin{prop}\label{prop_tilting}
The derived equivalence induces 
\[
\mr{mod}(J(Q,W)_C)_k
\simeq
\mr{mod}(J(Q',W')_{C'})^k.
\]
\end{prop}
\begin{proof}
First, take $E\in \mr{mod}(J(Q,W)_C)_k$.
For a composition $ba$ through $k$, $[bc]$ vanishes on $\Phi(E)$ since $a$ vanishes on $E$.
Hence we have $\Phi(E)\in \mr{mod}(J(Q',W')_{C'})^k$.

Next, assume that $\Phi(E)\in \mr{mod}(J(Q',W')_{C'})^k$ for $E\in \mr{mod}(J(Q,W))_k$.
For any $a$ with $t(a)=k$, $a$ vanishes on $E$ since 
$\sum_{s(b)=k}[ba]$ vanished and $\sum_{s(b)=k}b$ is injective. 
Hence we have $E\in \mr{mod}(J(Q,W)_{C})_k$.
\end{proof}
\begin{rem}
This gives a generalization of a part of the results of \cite{bridgeland-stern}.
\end{rem}
\end{NB}

\bibliographystyle{amsalpha}
\bibliography{bib-ver6.bib}

{\tt
\noindent Kentaro Nagao

\noindent Graduate School of Mathematics, Nagoya University

\noindent kentaron@math.nagoya-u.ac.jp
}
\end{document}